\documentclass{amsart}
\usepackage{hyperref}

\usepackage{amsrefs}

\usepackage{amsmath,amssymb,amsthm,amsfonts}

\usepackage{color}

\newtheorem{lemma}{Lemma}[section]
\newtheorem{theorem}{Theorem}[section]

\newtheorem{proposition}{Proposition}[section]
\newtheorem{remark}{Remark}[section]
\newtheorem{corollary}[theorem]{Corollary}

\numberwithin{equation}{section}

\newcommand{\Ndim}{n}

\newcommand{\dis}{\displaystyle}

\newcommand{\rmi}{{\rm i}}
\newcommand{\rmre}{{\rm Re}}

\newcommand{\id}{i}

\newcommand{\R}{\mathbb{R}}

\newcommand{\semiG}{\mathbb{A}}
\newcommand{\highG}{\Theta}
\newcommand{\highB}{\Lambda}
\newcommand{\sourceG}{g}
\newcommand{\solU}{f}
\newcommand{\wN}{\ell}

\DeclareMathOperator{\esssup}{ess\,sup}

\newcommand{\Ac}{A}

\newcommand{\vel}{v}
\newcommand{\dK}{K}

\newcommand{\sph}{{\mathbb S}^{\Ndim-1}}

\newcommand{\CurveP}{\zeta}
\newcommand{\Ctheta}{\vartheta}

\newcommand{\NgE}{{K \ge 2\ksob}}
\newcommand{\DgE}{\Ndim \ge 3}

\newcommand{\FM}{\mu}
\newcommand{\MM}{\sqrt{\mu}}

\newcommand{\FP}{\mathbf{P}}
\newcommand{\FL}{L}
\newcommand{\FI}{\mathbf{I}}

\newcommand{\CD}{\mathcal{D}}
\newcommand{\CE}{\mathcal{E}}
\newcommand{\CF}{\mathcal{F}}
\newcommand{\noCD}{\CD^0}

\newcommand{\CH}{\mathcal{H}}
\newcommand{\CHd}{\dot{\mathcal{H}}}

\newcommand{\CL}{\mathcal{L}}
\newcommand{\CN}{N(\FL)}

\newcommand{\AsoftI}{p(\Ndim)}
\newcommand{\Asoftq}{q(\Ndim)}
\newcommand{\jN}{j(\Ndim)}
\newcommand{\lN}{\ell(\Ndim)}

\newcommand{\na}{\nabla}
\newcommand{\al}{\alpha}
\newcommand{\be}{\beta}
\newcommand{\la}{\lambda}
\newcommand{\de}{\delta}
\newcommand{\pa}{\partial}
\newcommand{\ka}{\kappa}
\newcommand{\eps}{\epsilon}

\newcommand{\Ga}{\Gamma}

\newcommand{\eqdef}{\overset{\mbox{\tiny{def}}}{=}}
\newcommand{\opGstar}{\Gamma_{*}}

\newcommand{\weT}{q}

\newcommand{\teePLUSop}{T^{k}_{+}}
\newcommand{\teeSTARop}{T^{k}_{*}}

\newcommand{\Rap}{e}
\newcommand{\ArbI}{m}

\newcommand{\fcnP}{\hat{\rho}}

\newcommand{\wE}{j}
\newcommand{\wB}{b}

\newcommand{\SB}{\textbf{B}}

\newcommand{\threed}{{\mathbb R}^{\Ndim}}
\newcommand{\nsm}{|}

\newcommand{\ang}[1]{ \left< {#1} \right> }

\newcommand{\domain}{\mathbb{T}^\Ndim}

\newcommand{\ind}{ {\mathbf 1}}

\newcommand{\spacen}{N^{s,\gamma}}
\newcommand{\spaceELLn}{N^{s,\gamma}_\ell}
\newcommand{\ksob}{K^*_n}

\begin{document}

\title[Optimal time decay of the Boltzmann equation in $\threed_x$]
{Optimal time decay of the non cut-off Boltzmann equation in the whole space}

\author[R. M. Strain]{Robert M. Strain}
\address{
University of Pennsylvania, Department of Mathematics, David Rittenhouse Lab, 209 South 33rd Street, Philadelphia, PA 19104-6395, USA}
\email{strain at math.upenn.edu}
\urladdr{http://www.math.upenn.edu/~strain/}
\thanks{R.M.S. was partially supported by the NSF grant DMS-0901463, and an Alfred P. Sloan Foundation Research Fellowship.	}

\begin{abstract}
In this paper we study the large-time behavior of perturbative classical solutions to the hard and soft potential Boltzmann equation without the angular cut-off assumption in the whole space $\threed_x$ with $\DgE$.    We use the existence theory of global in time nearby Maxwellian solutions from \cite{gsNonCutA,gsNonCut0}.  
It has been a longstanding open problem to determine the large time decay rates for the soft potential Boltzmann equation in the whole space, with or without the angular cut-off assumption \cite{MR677262,MR2847536}.  For perturbative initial data, we prove that solutions converge to the global Maxwellian with the optimal large-time decay rate of  $O(t^{-\frac{\Ndim}{2}+\frac{\Ndim}{2r}})$ in the $L^2_\vel(L^r_x)$-norm for any $2\leq r\leq \infty$. 
\end{abstract}



%
\setcounter{tocdepth}{1}
\maketitle
\tableofcontents

\thispagestyle{empty}





\section{Introduction and main results}

In recent work the Boltzmann equation has been shown for the first time to have global in time classical perturbative solutions for physically realistic collision kernels in the case of the torus $(\domain_x)$ \cite{gsNonCutA,gsNonCut0} with $\Ndim \ge 2$ and in the whole space $(\R^3_x)$ case \cite{MR2847536}.  These results were able to successfully remove the  widespread (non-physical) ``Grad angular cut-off'' assumption in the context of perturbations.  The solutions on the torus exhibit exponential time decay $O(e^{-\la t})$ to equilibrium for the hard potentials and rapid polynomial decay  $O( t^{-k})$ for any $k>0$ in the case of the soft potentials.  This large time behavior is predicted by the celebrated Boltzmann H-theorem.  In the whole space, the presence of dispersion shackles the H-theorem to low order polynomial rates.  Even with angular cut-off \cite{MR677262}, it has remained a longstanding open problem to determine the  time decay rates for the soft potential kernels in $\threed_x$.  In this work, we prove optimal large-time decay rates for the full range of hard and soft potential collision kernels without angular cut-off.

We will study solutions to the {\em Boltzmann equation}, which is given by
  \begin{equation}
  \frac{\partial F}{\partial t} + v \cdot \nabla_x F = {\mathcal Q}(F,F),
  \quad
  F(0,x,\vel) = F_0(x,\vel).
  \label{BoltzFULL}
  \end{equation}
Here the unknown is $F=F(t,x,v)\ge 0$.  For each time $t\ge 0$, $F(t, \cdot, \cdot)$ represents the density of particles in phase space. The spatial coordinates we consider  are $x\in\threed_x$, and the velocities are $\vel\in\threed_v$ with 
$\Ndim \ge 2$.   The  {\em Boltzmann collision operator}, ${\mathcal Q}$,
is a bilinear operator which acts only on the velocity variables, $\vel$, as 
  \begin{equation*}
  {\mathcal Q} (G,F)(v) \eqdef 
  \int_{\mathbb{R}^n}  dv_* 
  \int_{\mathbb{S}^{n-1}}  d\sigma~ 
  B(v-v_*, \sigma) \, 
  \big[ G'_* F' - G_* F \big].
  \end{equation*} 
Here we are using the standard shorthand $F = F(v)$, $G_* = G(v_*)$, $F' = F(v')$, $G_*^{\prime} = G(v'_*)$. 
In this expression,  $v$, $v_*$ and $v'$, $v' _*$  are 
the velocities of a pair of particles before and after collision.  They are connected through the formulas
  \begin{equation}
  v' = \frac{v+v_*}{2} + \frac{|v-v_*|}{2} \sigma, \qquad
  v'_* = \frac{v+v_*}{2} - \frac{|v-v_*|}{2} \sigma,
  \qquad \sigma \in \mathbb{S}^{n-1}.
\notag 
  \end{equation}
The {\em Boltzmann collision kernel}, $B(v-v_*, \sigma)$, depends upon the {\em relative velocity} $|v-v_*|$ and on
the {\em deviation angle}  $\theta$ through 
$\cos \theta = ( v-v_* )\cdot \sigma/|v-v_*|$.  Without restriction we suppose that  $B(v-v_*, \sigma)$
is supported on $\cos \theta \ge 0$, as in \cite{MR1379589,gsNonCut0}.

\subsection*{The Collision Kernel}

Our assumptions are the following:
 \begin{itemize}
 \item 
We suppose that $B(v-v_*, \sigma)$ takes product form in its arguments as
\begin{equation}
B(v-v_*, \sigma) =\Phi( |v-v_*| ) \, b(\cos \theta).
\notag
\end{equation}
It generally holds that both $b$ and $\Phi$ are non-negative functions. 
 
  \item The angular function $t \mapsto b(t)$ is not locally integrable; for $c_b >0$ it satisfies
\begin{equation}
\frac{c_b}{\theta^{1+2s}} 
\le 
\sin^{\Ndim -2} \theta  ~ b(\cos \theta) 
\le 
\frac{1}{c_b\theta^{1+2s}},
\quad
s \in (0,1),
\quad
   \forall \, \theta \in \left(0,\frac{\pi}{2} \right].
   \label{kernelQ}
\end{equation}
 
  \item The kinetic factor $z \mapsto \Phi(|z|)$ satisfies for some $C_\Phi >0$
\begin{equation}
\Phi( |v-v_*| ) =  C_\Phi  |v-v_*|^\gamma , \quad \gamma  \ge -2s.
\label{kernelP}
\end{equation}
In the rest of this paper these will be called ``hard potentials.'' 

  \item Our results will also apply to the more singular situation
\begin{equation}
\Phi( |v-v_*| ) =  C_\Phi  |v-v_*|^\gamma , \quad  -2s > \gamma  > -n. 
\label{kernelPsing}
\end{equation}
These will be called ``soft potentials'' throughout this paper.

 \end{itemize}
 
These collision kernels are physically motivated since they can be derived from a spherical intermolecular repulsive potential such as 
$
\phi(r)=r^{-(p-1)}
$
with $p \in (2,+\infty)$ as was shown by Maxwell in 1866.  In the physical dimension $(n=3)$,  $B$ satisfies the conditions above with $\gamma = (p-5)/(p-1)$ and $s = 1/(p-1)$;  see \cite{MR1942465}.  A large amount of previous work requires the Grad angular cut-off assumption which usually means 
 either $b(\cos\theta) \in L^\infty(\sph)$ or $b(\cos\theta) \in L^1(\sph)$.  However neither of these assumptions are satisfied for angular factors such as \eqref{kernelQ}.
 
 We will study the linearization of \eqref{BoltzFULL} around the Maxwellian equilibrium states 
\begin{equation}
F(t,x,v) = \FM(v)+\sqrt{\FM(v)} f(t,x,v),
\label{maxLIN}
\end{equation}
where without loss of generality the Maxwellian is given by
$$
\FM(v) \eqdef (2\pi)^{-n/2}e^{-|v|^2/2}.
$$
We linearize the Boltzmann equation \eqref{BoltzFULL} around \eqref{maxLIN}.  This grants an equation for the perturbation, $f(t,x,v)$, that is given by
\begin{gather}
 \partial_t f + v\cdot \nabla_x f + \FL (f)
=
\Ga (f,f),
\quad
f(0, x, v) = f_0(x,v),
\label{Boltz}
\end{gather}
where the {\it linearized Boltzmann operator}, $\FL$, is defined as
$$
 \FL(g)
 \eqdef  
- \FM^{-1/2}\mathcal{Q}(\FM ,\MM g)- \FM^{-1/2}\mathcal{Q}(\MM g,\FM),
$$
and the bilinear operator, $\Gamma$, is then
\begin{gather}
\Gamma (g,h)
\eqdef
 \FM^{-1/2}\mathcal{Q}(\MM g,\MM h).
\label{gamma0}
\end{gather}
The $(\Ndim+2)$-dimensional null space of $\FL$ is well known \cite{MR1379589}: 
\begin{equation}\label{nullLsp}
    \CN \eqdef {\rm span}\left\{\sqrt{\FM}, ~  \vel_1\sqrt{\FM}, ~\ldots, ~  \vel_\Ndim\sqrt{\FM}, ~   (|\vel|^2-\Ndim)\sqrt{\FM} \right\}.
\end{equation}
Now, for fixed $(t,x)$, we define the orthogonal projection from
$L^2_\vel$ to $\CN$ as
\begin{equation}
 \FP f=a^f(t,x)\sqrt{\FM}
    +\sum_{i=1}^{\Ndim}
    b^f_i(t,x)\vel_i\sqrt{\FM}+c^f(t,x)(|\vel|^2-\Ndim)\sqrt{\FM}, \label{form.p}
\end{equation}
where the functions $a^f$, $b^f\eqdef [b_1,\cdots,b_n]$ and $c^f$ will depend on $\solU(t,x,\vel)$.

 Our main interest is in the large-time behavior of global classical solutions of the Cauchy problem for the Boltzmann equation \eqref{BoltzFULL} which are perturbations of the Maxwellian equilibrium states \eqref{maxLIN} for the long-range collision kernels \eqref{kernelQ}, \eqref{kernelP} and \eqref{kernelPsing}.  This behavior is controlled by the celebrated Boltzmann $H$-theorem. 
 
We define the Boltzmann H-functional by
$$
{H}(t)\eqdef-\int_{\threed}dx~ \int_{\threed}dv~ f(t,x,\vel) \log f(t,x,\vel).
$$
Then the Boltzmann H-theorem predicts that, for solutions of the Boltzmann equation, the entropy is increasing over time;  this corresponds to the formal statement
$$
\frac{d{H}(t)}{dt} = \int_{\threed} dx ~ D(f,f)  \ge 0.
$$
This is a part of the second law of thermodynamics.    Here $D(f,f)$ is the well-known ``entropy production functional'' which does not operate on $(t,x)$.  This shows that the H-theorem is degenerate in $(t,x)$ because characterizations of 
$D(f,f)$ will be local in those variables,
and so the transport terms in \eqref{BoltzFULL} become important.  Of course this is an extremely formal calculation, even more so in the whole space where the H-functional is infinity at any non-zero global Maxwellian.

In recent work \cite{gsNonCut0,gsNonCutA},
Gressman and the author have introduced into the Boltzmann theory the 
following sharp weighted geometric fractional  Sobolev norm:
\begin{equation} 
\nsm f\nsm_{\spacen}^2 
\eqdef 
\nsm f\nsm_{L^2_{\gamma+2s}}^2 + 
\int_{\mathbb{R}^n} dv \int_{\mathbb{R}^n} dv' ~
(\ang{v}\ang{v'})^{\frac{\gamma+2s+1}{2}}
 \frac{(f' - f)^2}{d(v,v')^{n+2s}} 
\ind_{d(v,v') \leq 1}. \label{normdef} 
\end{equation}
Generally, $\ind_{A}$ is the standard indicator function of the set $A$.
Now this space includes the weighted $L^2_\wN$ space, for $\wN\in\mathbb{R}$, with norm given by
$$
\nsm f\nsm_{L^2_{\wN}}^2 
\eqdef
\int_{\mathbb{R}^n} dv ~ 
\ang{v}^{\wN}
~
|f(v)|^2.
$$
The weight is
$
\ang{v}
\eqdef \sqrt{1+|v|^2}.
$
The fractional differentiation effects are measured 
using the anisotropic metric $d(v,v')$ on the ``lifted'' paraboloid (in $\R^{\Ndim+1}$) as
$$
d(v,v') \eqdef \sqrt{ |v-v'|^2 + \frac{1}{4}\left( |v|^2 -  |v'|^2\right)^2}.
$$
This metric encodes the nonlocal anisotropic changes in the power of the weight. 

 The linearized collision operator $\FL$ is non-negative and it is furthermore locally coercive in
the sense that there is a constant $\la>0$ such that \cite[Theorem 8.1]{gsNonCut0}:  
\begin{equation}\label{coerc}
\rmre \ang{g, \FL g}
=
\rmre \int_{\threed} dv ~ g(v) \cdot \overline{L(g)}(v)
\geq 
\la\nsm\{\FI-\FP\}g\nsm^2_{\spacen}.
\end{equation}
This may be interpreted loosely as a linearized statement of the H-theorem.  Note that the norm $\spacen$ provides a sharp  characterization of the linearized collision operator \cite[(2.13)]{gsNonCut0}; in earlier work \cite{MR2322149} the sharp gain of velocity weight in $L^2$ was established for the non-derivative part of \eqref{normdef}.  $\spacen$ also controls sharply the nonlinear collision operator, and its entropy production estimates \cite{gsNonCutEst} for $D(f,f)$.  However all of these coercive estimates are degenerate in the $(t,x)$ variables.

In a bounded domain such as the torus ($\domain_x$), the H-theorem is prominent and then rapid convergence can be established (as in \cite{gsNonCutA,gsNonCut0}).  However in the whole space ($\threed_x$) the dispersive effects dominate the H-theorem, and so the transport terms in \eqref{Boltz} restrict the convergence to low order polynomial rates.  This point of view illustrates the additional difficulty involved in proving decay rates in the whole space.
We state these time decay rates in our main theorems of Section \ref{sec:main}.  Prior to that we introduce the notation.

\subsection{Notation}

 For any $m\geq 0$, we use $H^m$ to denote the
usual Sobolev spaces $H^m(\threed_x\times\threed_\vel)$, $H^m(\threed_x)$, or  $H^m(\threed_\vel)$, respectively, where for example
$$
H^m_{\ell}(\threed_\vel)
\eqdef
\left\{ f \in L^2_\ell (\threed_\vel) : \nsm f \nsm_{H^m_{\ell}(\threed)}^2 
\eqdef \int_{\threed} dv \ang{v}^\ell \left| (I - \Delta_v)^{m/2} f(v) \right|^2 <\infty
\right\}.
$$
Then we denote $H^m_{0} = H^m$.
For a Banach space $X$,
we let $\|\cdot\|_{X}$ denote the corresponding norm over $X(\threed_x\times\threed_\vel)$, 
and 
$\nsm\cdot\nsm_{X}$ analogously denotes the norm only over $X(\threed_\vel)$.  For example we use the notation
$$
\| h\|_{\spacen}^2
=
\| h\|_{\spacen(\threed_x\times\threed_\vel)}^2
\eqdef
\left\|~ \nsm h\nsm_{\spacen(\threed_\vel)} ~ \right\|_{L^2(\threed_x)}^2.
$$
Furthermore $X(B)$ denotes the Banach space $X$ over the domain $B\subset \threed$.  In particular, we will use the notation $B_R$ to denote the $\Ndim$ dimensional ball of radius $R>0$ centered at the origin.  
Sometimes further $L^2_x$ and $L^2_\vel$ are
used to denote $L^2(\threed_x)$ and $L^2(\threed_\vel)$
respectively.  There should be no confusion between $L^2_x$, $L^2_\vel$ and $L^2_\ell$, etc, since $x$ and $\vel$ are never used to denote a weight.

For an integrable function $g: \threed\to\R$, its Fourier transform is defined by
\begin{equation*}
  \widehat{g}(k)= \CF g(k)\eqdef \int_{\threed} e^{-2\pi \rmi x\cdot k} g(x)dx, \quad
  x\cdot
   k\eqdef\sum_{j=1}^\Ndim x_jk_j,
   \quad
   k\in\threed,
\end{equation*}
where $\rmi =\sqrt{-1}\in \mathbb{C}$. For two
complex vectors $a,b\in\mathbb{C}^\Ndim$, $(a\mid b)=a\cdot
\overline{b}$ denotes the dot product over the complex field, where
$\overline{b}$ is the ordinary complex conjugate of $b$.

We use $\langle\cdot,\cdot\rangle$ to denote the inner product over the Hilbert space $L^2_\vel$, i.e.
\begin{equation*}
    \langle g,h\rangle=\int_{\threed} g(\vel)\cdot \overline{h(\vel)} ~ d\vel,\ \ g=g(\vel), ~h=h(\vel)\in
    L^2_\vel.
\end{equation*}
Analogously $\left(\cdot,\cdot\right)$ denotes the inner product over $L^2(\threed_x\times\threed_\vel)$.
For $r\geq 1$, we define the mixed Lebesgue
space $Z_r=L^2_\vel(L^r_x)=L^2(\threed_\vel;L^r(\threed_x))$ with the norm
\begin{equation*}
\|g\|_{Z_r}\eqdef \left(\int_{\threed}\left(\int_{\threed}
    |g(x,\vel)|^rdx\right)^{2/r}d\vel\right)^{1/2}.
\end{equation*}
We introduce the norms
$\|\cdot\|_{\CHd^m}$ and $\|\cdot\|_{\CH^m}$ with $m\geq 0$ given by
\begin{equation}\label{brief.norm}
    \|\solU\|_{\CHd^m}^2\eqdef \|\solU\|_{L^2_\vel(\dot{H}^m_x)}^2,
    \quad
    \|\solU\|_{\CH^m}^2\eqdef \|\solU\|_{L^2_\vel(H^m_x)}^2,
    \quad
    \|\solU\|_{\CL^2}^2 \eqdef\|\solU\|_{\CH^0}^2.
\end{equation}
Here $\dot{H}^m_x=\dot{H}^m(\threed_x)$ is the standard homogeneous $L^2_x$ based Sobolev space:
$$
\| g \|_{\dot{H}^m(\threed_x)}^2 \eqdef \int_{\threed} dk ~ |k|^{2m} | \hat{g}(k)|^2.
$$
We also define the unified weight function as follows
\begin{equation}
w(v) 
\eqdef 
\left\{
\begin{array}{ccc}
\ang{v}, &\gamma + 2s \ge 0, & \text{hard potentials: \eqref{kernelP}}, \\
\ang{v}^{-\gamma - 2s}, & \gamma + 2s < 0, & \text{soft potentials: \eqref{kernelPsing}.}\\
\end{array}
\right.
\label{weigh}
\end{equation}
We then consider the weighted anisotropic derivative space as in \eqref{normdef}: 
 \begin{gather*}
\nsm h \nsm^2_{{\spaceELLn}}
\eqdef 
\nsm w^{\ell} h\nsm_{L^2_{\gamma+2s}}^2 + \int_{\mathbb{R}^n} dv ~ 
\ang{v}^{\gamma+2s+1} w^{2\ell}(v)
\int_{\mathbb{R}^n} dv' 
~
 \frac{(h' - h)^2}{d(v,v')^{n+2s}} 
\ind_{d(v,v') \leq 1}.
\end{gather*}
Note that  
$
\nsm h \nsm_{{\spacen}}
=
\nsm h \nsm_{{\spacen_0}}.
$
For
multi-indices, we denote
\begin{equation*}
 \pa^{\al}_\be=\pa_{x_1}^{\al_1}\cdots\pa_{x_\Ndim}^{\al_\Ndim}
    \pa_{\vel_1}^{\be_1}\cdots\pa_{\vel_\Ndim}^{\be_\Ndim},
    \quad
    \al=[\al_1,\ldots,\al_\Ndim],
    \quad
\be=[\be_1,\ldots,\be_\Ndim].
\end{equation*}
The length of $\al$ is $|\al|=\al_1+\cdots+\al_\Ndim$ and the length of
$\be$ is $|\be|=\be_1+\cdots+\be_\Ndim$.

Fix $\wN\ge 0$.
Given a solution, $\solU(t,x,\vel)$, to the Boltzmann equation \eqref{Boltz}, we define an instant energy functional to be a continuous function, 
 $\CE_{\dK,\wN}(t)$, which satisfies
\begin{equation}
\CE_{\dK,\wN}(t)\approx
\sum_{|\al|+|\be|\leq \dK}\|w^{\wN-|\be|}\pa^\al_\be \solU(t)\|_{\CL^2}^2.
\label{def.eNm}
\end{equation}
Also define the high-order instant energy functional $\CE_{\dK,\wN}^{\rm h}(t)$ as
\begin{equation}
\CE_{\dK,\wN}^{\rm h}(t)\approx
 \sum_{1\leq |\al|\leq \dK}\|w^{\wN}\pa^\al \solU(t)\|_{\CL^2}^2+\sum_{|\al|+|\be|\leq \dK}\|w^{\wN-|\be|}\pa^\al_\be  \{\FI-\FP\}\solU(t)\|^2_{\CL^2}.
 \label{def.eNm.h}
\end{equation}
We furthermore define the dissipation rate $\CD_{\dK,\wN}(t)$ as
\begin{equation}
\CD_{\dK,\wN}(t)
\eqdef
\sum_{1\leq |\al|\leq \dK}\|\pa^\al \solU(t)\|_{\spacen_\wN}^2+\sum_{|\al|+|\be|\leq \dK}\|\pa^\al_\be \{\FI-\FP\}\solU(t)\|_{\spacen_{\wN-|\be|}}^2.
\label{def.dNm}
\end{equation}
For brevity, when $\wN=0$ we write $\CE_{\dK}(t)=\CE_{\dK,0}(t)$, $\CE_{\dK}^{\rm h}(t)=\CE_{\dK,0}^{\rm h}(t)$ and $\CD_{\dK}(t)=\CD_{\dK,0}(t)$.  We suppose once and for all that $\dK$ is an integer satisfying $\NgE$, where $\ksob \eqdef \lfloor \frac{n}{2} +1 \rfloor$ is the smallest integer which is strictly greater than $\frac{\Ndim}{2}$.

Throughout this paper  we let $C$  denote
some positive (generally large) inessential constant and $\la$ denotes some positive (generally small) inessential constant, where both $C$ and
$\la$ may change values from line to line. 
Furthermore $A \lesssim B$ means $A \le C B$, and 
$A \gtrsim B$ means $B \lesssim A$.
In addition,
$A\approx B$ means $A \lesssim B$ and $B \lesssim A$.

\subsection{Main results}\label{sec:main}
In this subsection we state our main optimal time decay results for the Boltzmann equation \eqref{Boltz}.  We begin with the following existence result, and Lyapunov inequalities, based on the theory from \cite{gsNonCut0}.

\begin{theorem}\label{thm.energy}
Fix $\wN \ge 0$ and $\solU_0(x,\vel)$.  
There are $\CE_{\dK,\wN}(t)$, $\CD_{\dK,\wN}(t)$ such that if $\CE_{\dK,\wN}(0)$ is
sufficiently small, then the Cauchy problem to
the Boltzmann equation \eqref{Boltz} admits a unique global solution
$\solU(t,x,\vel)$ satisfying the Lyapunov inequality
\begin{eqnarray}
\frac{d}{dt}\CE_{\dK,\wN}(t)+\la \CD_{\dK,\wN}(t)\leq 0,
\quad \forall t\geq 0.
\label{thm.energy.1}
\end{eqnarray}
Here $\la >0$ may depend on $\wN$.  In addition there is $\CE_{\dK,\wN}^{\rm h}(t)$
such that  
\begin{equation}\label{thm.energy.2}
\frac{d}{dt}\CE_{\dK,\wN}^{\rm h}(t)+\la \CD_{\dK,\wN}(t)
\lesssim
 \|\na_x\FP \solU(t)\|_{\CL^2}^2,
\quad \forall t\geq 0.
\end{equation}
Furthermore, if
$
 F_{0}
 =\FM+\MM\solU_{0}
 \geq 0,
$
then
$
F(t,x,\vel)
=\FM+\MM\solU
(t,x,\vel)
\geq 0.
$
\end{theorem}

The nonlinear energy estimates in Theorem \ref{thm.energy} together with time-decay estimates on the linearized
 system in Theorem \ref{thm.ls} lead us to the time-decay rates of
the instant energy functionals $\CE_{\dK,\wN}(t)$,
$\CE_{\dK}^{\rm h}(t)$ and the $Z_r$ norms  if we make additional integrability assumptions on the initial data. Precisely, for $\wN\ge0$, set $\eps_{\dK,\wN}$ 
to be
\begin{equation}\label{def.id.jm}
    \eps_{\dK,\wN}\eqdef\CE_{\dK,\wN}(0)+\|\solU_0\|_{Z_1}^2.
\end{equation}
Our main optimal time decay result is as follows:

\begin{theorem}\label{thm.ns}
Let $\solU(t,x,\vel)$ be the solution to the Cauchy problem of the Boltzmann equation \eqref{Boltz} obtained in Theorem \ref{thm.energy}. Suppose $\eps_{\dK,\wN+ \AsoftI}$ is sufficiently small where 
$\AsoftI =0$ for the hard potentials \eqref{kernelP} and for the soft potentials \eqref{kernelPsing}
we take any number $\AsoftI> \frac{\Ndim}{2}$.  Then for any $t\geq 0$ we uniformly have
\begin{equation}\label{thm.ns.2}
    \CE_{\dK,\wN}(t)\lesssim \eps_{\dK,\wN+\AsoftI} (1+t)^{-\frac{\Ndim}{2}}.
\end{equation}
Suppose further that $\eps_{\dK,\lN+\AsoftI}$ is sufficiently small, 
where the weight factor $\lN$  is defined, for any small $\varepsilon>0$,  as follows
\begin{equation} 
\lN
\eqdef 
\left\{
\begin{array}{ccc}
2(\gamma+2s), & \text{ for the hard potentials: \eqref{kernelP}}, \\
 \frac{\Ndim}{2}+\ksob+1 + \varepsilon, &  \text{for the soft potentials: \eqref{kernelPsing}.}\\
\end{array}
\right.
\label{exponent.ln}
\end{equation}
Then for any $2\leq r\leq \infty$, we have the following  estimate:  
\begin{equation} 
 \|\solU(t)\|_{Z_r}\lesssim(1+t)^{-\frac{\Ndim}{2}+\frac{\Ndim}{2r}},
 \label{cor.1}
\end{equation}
which holds uniformly over $t \ge 0$.
\end{theorem}

 Furthermore, we have faster time decay rates for higher derivatives and special components of the solution as follows.

\begin{corollary}\label{cor}
Let $\solU(t,x,\vel)$ be the solution to the Cauchy problem \eqref{Boltz} of the Boltzmann equation from Theorem \ref{thm.energy}.  If, in addition, $\eps_{\dK,\wN+\Asoftq}$ is sufficiently small then we have uniformly in $t\geq 0$ the time decay estimate
\begin{equation}\label{thm.ns.3}
    \CE_{\dK,\wN}^{\rm h}(t)  \lesssim  \eps_{\dK,\wN+\Asoftq} (1+t)^{-\frac{\Ndim+2}{2}+\varepsilon}.
\end{equation}
Here for the soft potentials, \eqref{kernelPsing}, for any small $\varepsilon > 0$ we choose 
$\Asoftq =\Asoftq(\varepsilon)$ sufficiently large.  For the hard potentials, \eqref{kernelP}, then we can choose  $\varepsilon = \Asoftq= 0$.

If both $\eps_{\dK,\Asoftq}$ and $\eps_{\dK,\lN+\AsoftI}$ are sufficiently small, 
where $\lN$ is defined in \eqref{exponent.ln}, then
 for any $2\leq r\leq \infty$, we have the following time decay estimate:  
\begin{equation}
  \|\{\FI-\FP\}\solU(t)\|_{Z_r}\lesssim(1+t)^{-\frac{\Ndim}{2}
 -\frac{1}{r}+\frac{\Ndim}{2r}+\varepsilon}.\label{cor.2}
\end{equation}
This will hold uniformly for any $t\geq 0$.    Again for the soft potentials, \eqref{kernelPsing}, we use any small $\varepsilon > 0$.  But the hard potentials, \eqref{kernelP}, we can choose  $\varepsilon = 0$.
\end{corollary}

These time decay rates for the $L^2$-norms in \eqref{thm.ns.2} and \eqref{cor.1} 
are optimal in the sense that they are the same as those for the linearized system, as in Theorem \ref{thm.ls}, which is studied using Fourier analysis.  These rates also coincide with those in the case of the Boltzmann equation \cite{MR0363332} for hard-sphere particles, and further they are the same in $L^r_x$ as those for solutions to the Heat equation. Corollary \ref{cor} shows that higher derivatives and the microscopic part of the solution decays faster.

\subsection{Historical discussion, new methods, and future directions}
Now our main theorems and their corollary show that the Cauchy problem for the non cut-off Boltzmann equation \eqref{Boltz} is {\it Hypocoercive} for perturbations, this holds in the sense of the description given by Villani \cite{villani-2006}.

We would like to point out that there have been extensive investigations on the rate of convergence to Maxwellian equilibrium for the nonlinear Boltzmann equation or related kinetic equations.  We only have the space to mention a brief few.   In the context of perturbations,  it was Ukai \cite{MR0363332}
who in 1974 proved the first decay result.  Here the spectral analysis was used to obtain the exponential rates for the
Boltzmann equation with hard potentials on torus.  Further time decay results on the torus were obtained in \cite{MR2116276,arXiv:0912.1742,MR575897,MR2209761,MR2366140,strainSOFT} and the references therein.

In particular we have studies of the decay rates for the Vlasov-Poisson-Boltzmann \cite{arXiv:0912.1742} and Vlasov-Maxwell-Boltzmann \cite{MR2209761} system using the existence theory from \cite{MR2000470}.  
We also mention the interaction functional approach from Duan \cite{MR2420519} which removes the time derivatives.
We further have decay results with an angular cut-off for the moderate soft potentials in \cite{MR575897} and for the full range of soft potentials in \cite{MR2209761,MR2366140}.  Now there are also rapid decay results for the relativistic Boltzmann equation \cite{strainSOFT}.  Of course these methods apply rigorously in the perturbative regime.

With the entropy production method Desvillettes-Villani \cite{MR2116276} obtained the first almost exponential
rate of convergence for solutions to the Boltzmann equation on the torus
with cut-off soft potentials for initial data without a size restriction assuming additional global in time uniform 
high regularity and moment bounds and  good lower bounds for the density at high velocities.  Then Strain-Guo \cite{MR2366140}
provided a very simple proof of the main decay results in \cite{MR2116276} for the unconditional perturbative regime, assuming only the high moment bounds without extra regularity.


To study the optimal convergence rates in the whole space has proven to be harder than the case of the torus because of the additional dispersive effects of the transport term in \eqref{Boltz}.  
The early results are well documented in Glassey \cite{MR1379589}.  

In particular we mention the Kawashima \cite{MR1057534} method of thirteen moments, also the optimal linear decay analysis of Duan, Ukai, Yang, Zhao \cite{MR2357430} for the hard sphere case.  
Also \cite{D-Hypo} studies in particular the Boltzmann equation with confining forces.
Further recently we have seen proofs of the optimal time decay for the one-species Vlasov-Poisson-Boltzmann system \cite{arXiv:0912.1742}  in $\R^3_x$, and the two-species Vlasov-Maxwell-Boltzmann system \cite{2010arXiv1006.3605D} using the existence theory from \cite{MR2259206}.
In the spirit of the Kawashima's work \cite{MR1057534}, for the linearized time-decay  analysis, instead of using the compensation function as in \cite{MR1057534}, a key idea in \cite{D-Hypo,arXiv:0912.1742,2010arXiv1006.3605D} is to design several interactive functionals in order to exploit  the dissipation which is present in the degenerate parts of the solution.  For further references and discussions of these and other related results we refer to the commentary in \cite{arXiv:0912.1742,2010arXiv1006.3605D}.

 We point out that the methods above do not apply to the Boltzmann equation, with or without angular cut-off, in the whole space for a soft potential.  In that context we only have the result of Ukai-Asano \cite{MR677262} from 1982.   
For the cut-off moderately soft potentials, so that roughly $b(\cos\theta)\in L^\infty(\sph)$ instead of \eqref{kernelQ} and $-1< \gamma \le 0$ instead of \eqref{kernelP} and \eqref{kernelPsing}, they obtain in the whole space
$$
\| \ang{\cdot}^{\wB} \solU(t,\cdot,\cdot) \|_{L^\infty_v(H^m_x)}
\lesssim (1+t)^{-a}.
$$
Above $\ang{\cdot}$ denotes $\ang{\vel}$ in the norm as usual.  Here $a \eqdef \min\left\{\frac{n}{2}\left(\frac{1}{p} - \frac{1}{2}\right), 1\right\}$ for $p\in [1,2)$.  
Their initial data uses that the following quantity is sufficiently small  
$$
\| \ang{\cdot}^{\wB+a \gamma} \solU_0(\cdot,\cdot) \|_{L^\infty_v(H^m_x)}
+ \| \solU_0\|_{Z_p},
$$
where $\wB > \frac{\Ndim}{2} - \gamma$ and $m > \frac{\Ndim}{2}$.
This convergence rate of $a=\frac{n}{4}$ when $p=1$ is optimal, in comparison to ours, in dimensions $\Ndim =2$, $3$, $4$ but the rate of $a=1$ is not optimal for $\Ndim \ge 5$.  Their methods involve the spectral analysis of the semi-group, which as far as we are aware has not been extended lower than $\gamma>-1$.  

At the same time, we would like to mention that after the main results in this article were complete, a related paper by Alexandre, Morimoto, Ukai, Xu, and Yang appeared in \cite{MR2847536}.  For the non cut-off soft potentials $\gamma + 2s <0$ in the range $\max\{-3, -2s -3/2\}< \gamma \le -2s$, they give the following convergence rate in \cite{MR2847536}:
$$
\esssup_{x\in \R^3_x} \nsm \solU(t,x) \nsm_{H^{\dK-3}(\R^3_\vel)}  \lesssim  (1+t)^{-1}.
$$
This non-optimal rate is proven on the basis of the pure energy method and a time differential inequality.
Note that in the existence theorem they use $\dK \ge 6$ derivatives and $\wN \ge \dK +1$ weights in $\Ndim = 3$ dimensions.  
They also prove some optimal convergence rates for the hard potential case when $\gamma + 2s > 0$. We point out that 
our optimal decay results will apply to their solutions  \cite{MR2847536}.

As far as we know, other than ours, these are the only two results obtaining time decay rates in the whole space for the soft potential Boltzmann equation with or without the angular cut-off assumption.

To obtain the results from this paper in Theorem \ref{thm.energy}, Theorem \ref{thm.ns}, and Corollary \ref{cor}, 
we build upon previous work developed in  collaborations of the author with Renjun Duan \cite{arXiv:0912.1742,2010arXiv1006.3605D}, Phillip T. Gressman \cite{gsNonCutA,gsNonCut0,gsNonCutEst}, and Yan Guo \cite{MR2209761,MR2366140}.  The ideas in those works were to estimate the dispersion in the whole space using the interactive functionals and the Fourier analysis \cite{arXiv:0912.1742,2010arXiv1006.3605D}, to prove sharp estimates for the Boltzmann collision operator without angular cut-off \cite{gsNonCutA,gsNonCut0,gsNonCutEst}, and also to prove rapid decay on the torus for the full range of cut-off soft potentials using interpolation 
\cite{MR2209761}
and splitting methods \cite{MR2366140}. 
 
 The present work incorporates a fusion of each of these different ideas, but it also requires several new methodologies.  
In particular previously the Fourier analysis techniques for studying the dispersion were designed around the hard-potential cases where the dissipation is as strong as the instant energy.  For the soft potentials this is just simply false.

To fix this difficulty, we need to prove new weighted instant time-frequency Lyapunov inequalities in $(t,k)$ on the Fourier transform side after integrating in $\vel$.  We then must show that a functional satisfying \eqref{thm.tfli.1}, i.e.
\begin{equation}\notag
    \CE_\wN(t,k) \approx \nsm w^\ell \hat{\solU}(t,k)\nsm_{L^2}^2,
\end{equation}
is preserved by the linear flow for any $\wN\in\R$.  Once we have this preservation, we can use interpolation with a family of higher weight functions to obtain the linear decay.   Yet to prove the preservation is not so obvious, in fact we find an appropriate functional $\CE_\wN(t,k)$ which must be defined in two different ways for $|k| \le 1$ and alternatively for $|k| >1$.  Fortunately different inequalities are available in each case that allow us to exploit the changing behavior of the solution in these two separate regimes.

However unfortunately this interpolation method fails by necessity at the non-linear level in the whole space because the dissipation \eqref{def.dNm} does not (and can not) contain the macroscopic components \eqref{form.p}.   To prove the nonlinear decay we use the time-velocity splitting.  Note that this splitting was designed to deduce exponential decay $O(e^{-\la t^p})$ for $p\in (0,1)$ in the case of a cut-off soft potential \cite{MR2366140} with an exponential velocity weight on the initial data.  This is a completely different purpose from the one for which we use the splitting herein.    For this paper, the novelty is in that the error terms can be controlled by extra polynomial velocity weights for the soft potentials \eqref{kernelPsing}, and the linear decay allows us to deduce the convergence rates in Theorem \ref{thm.ns} and Corollary \ref{cor} using the Lyapunov inequalities from Theorem \ref{thm.energy}.  

We also gain back a large ``weight loss'' nonlinearly since the microscopic part $\FP \solU$ decays exponentially in the velocity variables and the linear decay works for any $\wN \in \R$.  This gain is further aided by the fact that we prove new estimates for the $L^2_v$ norm of $\Ga(\solU, \solU)$ from \eqref{gamma0} which allow negative weights on the initial data.

To obtain the optimal time decay rates in $L^r_x$ with $2\le r \le \infty$ as in \eqref{cor.1}  we use  an optimized Sobolev inequality \eqref{lem.semi.inf.p1}.  Here we recall \cite{2010arXiv1006.3605D}.

Furthermore,  we expect that the methods developed in this paper can be useful in several other physical contexts.   We believe that our new approaches developed in this paper are generally applicable for proving time decay rates for soft-potential kinetic equations with perturbative initial data.
In particular, including numerous additional efforts, with Zhu we can prove the optimal time decay for the relativistic Boltzmann equation \cite{ZhuStrain} in the whole space.  Moreover, in particular, it may be interesting to check if these optimal decay results could also be carried out for the solutions to the Landau equation obtained by Guo \cite{MR1946444} in the whole space.  
We additionally believe that these methods can apply to several other various kinetic equations where the soft potentials are present.

Lastly we point out that these results are constructive in the sense that it is possible to track all of the constants, although we make no effort to do so.

\begin{remark}
Note that for simplicity we have used $\NgE$ derivatives in the above existence results.  However for the hard potentials \eqref{kernelP}, or more generally under \eqref{kernelPsing} combined with $\gamma + 2s > -\frac{\Ndim}{2}$, these decay results will apply under less stringent regularity assumptions, see \cite{gsNonCut0}.
\end{remark}

\begin{remark}
It is worth mentioning that all of our main results above hold in $\DgE$ dimensions.   Furthermore, Theorem \ref{thm.energy} holds under $\Ndim \ge 2$.  When $\Ndim =2$, it can be seen from the proofs in Section \ref{sec.decayNL}
that logarithms in the time decay rate show up as in the estimate \cite[Proposition 4.5]{strainSOFT}.   Thus easy modifications of our proofs would allow the same decay results as above when $\Ndim =2$, however in each case above we would lose a small epsilon in the decay rate as a result of the temporal log.  This could be upgraded to optimal decay by proving that for a given data $\solU_0$ which satisfies $\FP \solU_0 =0$, then the linear decay in Theorem \ref{thm.ls} is actually faster by $\frac{1}{2}$.  Such a property is well known for the Boltzmann equation \cite{MR677262,MR0363332} with an angular cut-off and hard or moderately soft-potentials. We are willing to conjecture that this property holds true even without angular cut-off, but proofs of such properties usually use the spectral analysis of the linearized collision operator.  We avoid a complicated spectral study by using the Fourier transform, and with such methods we are unaware of any proof of this property.    Note further that all of the the linear decay estimates in Section \ref{sec.decayl} hold true for any $\Ndim \ge 2$. 
\end{remark}

\subsection{Organization of the paper}
The rest of this paper is organized as follows.  In Section \ref{sec.decayl}, we will study the time decay of solutions to the linear Boltzmann equation \eqref{ls}.  This is decomposed into several steps which are outlined at the beginning of Section \ref{sec.decayl} below.  Then in Section \ref{sec.decayNL} we start out by proving the energy Lyapunov inequalities from Theorem \ref{thm.energy}, and we finish by proving the nonlinear time decay rates from Theorem \ref{thm.ns} and Corollary \ref{cor}.

\section{Linear time decay}
\label{sec.decayl}

In this section we study the time-decay properties of solutions to the Cauchy problem for the linearized non cut-off Boltzmann equation \eqref{Boltz}.
We state our main results in the first subsection.  Then in Section \ref{sec.sub.tfli} we derive several velocity weighted pointwise time-frequency Lyapunov inequalities.  A key point here is that we can include weights and prove pointwise instantaneous bounds for solutions to the linearized equation.  
Finally in Section \ref{sec.tf} the temporal decay rates of the solution and
its derivatives in $L^2_x$ are proven as in Theorem \ref{thm.ls} and Corollary \ref{cor.ls}.

\subsection{Time decay properties of solutions to the linearized equation}
We consider the linearized Boltzmann equation with a
microscopic source $\sourceG=\sourceG(t,x,\vel)$:
\begin{equation}\label{ls}
    \left\{\begin{array}{l}
  \dis     \pa_t \solU+\vel\cdot\na_x \solU +\FL \solU =\sourceG,\\
\dis \solU|_{t=0}=\solU_0.
    \end{array}\right.
\end{equation}
For the nonlinear system \eqref{Boltz}, the non-homogeneous source term is given by
\begin{equation}\label{def.g.non}
    \sourceG=\Ga(\solU,\solU).
\end{equation}
In this case $\sourceG=\{\FI-\FP\} \sourceG$.  Solutions of \eqref{ls} formally take the following form
\begin{equation}
    \solU(t)=\semiG(t)\solU_0 + \int_0^t ds ~ \semiG(t-s)~\sourceG(s), \quad \semiG(t)\eqdef e^{-t\SB}, \quad \SB \eqdef \FL+\vel\cdot\na_x.
    \label{ls.semi}
\end{equation}
Here $\semiG(t)$ is the linear solution operator for the Cauchy problem corresponding to  \eqref{ls} with $\sourceG=0$.
  The main result of this section is stated as follows.

\begin{theorem}\label{thm.ls}
Fix $1\leq r\leq 2$, $m \ge 0$,   and $\ell \in\R$.
Consider the Cauchy problem \eqref{ls} with $\sourceG=0$.  The solution of the
linearized homogeneous system satisfies 
\begin{equation}
 \|w^\wN \semiG(t)\solU_0\|_{\CHd^m}
\lesssim
(1+t)^{-\sigma_{r,m}}
\|w^{\wN}\solU_0\|_{\CHd^m\cap Z_r},
 \label{thm.ls.1}
\end{equation}
for the hard potentials \eqref{kernelP} and any $t\geq 0$.  Here $\sigma_{r,m}$ is given by
\begin{equation}
\label{rateLIN}
    \sigma_{r,m}\eqdef \frac{\Ndim}{2}\left(\frac{1}{r}-\frac{1}{2}\right)+\frac{m}{2}.
\end{equation}
For the soft potentials \eqref{kernelPsing} with $\wE > 2 \sigma_{r,m}$ we further obtain for any $t\geq 0$ that
\begin{equation}
 \|w^\wN \semiG(t)\solU_0\|_{\CHd^m}
\lesssim
(1+t)^{-\sigma_{r,m}}  \|w^{\wN+\wE}\solU_0\|_{\CHd^m \cap Z_r}.
 \label{thm.ls.1.soft}
\end{equation}
\end{theorem}

We remark that the $\CHd^m$ norm above is a convenient tool, which could be replaced by any $\partial^\al_x$ with $|\al | =m$.  On the basis of the previous theorem, we have the following corollary which allows faster linear time decay away from the null space \eqref{form.p}.

\begin{corollary}\label{cor.ls}
Under the same conditions as Theorem \ref{thm.ls} we have
\begin{equation}
 \|\{\FI - \FP\}\semiG(t)\solU_0\|_{\CHd^m}
\lesssim
(1+t)^{-\sigma_{r,m+1}+\epsilon}\|w^{\wE}\solU_0\|_{\CHd^{m+1}\cap Z_r}.
\notag
\end{equation}
For the hard potentials \eqref{kernelP} we take $\wE=0$ and $\epsilon =0$.   Considering the soft potentials \eqref{kernelPsing},
 for any $\epsilon>0$
 we choose $\wE = \wE(\epsilon)>0$ sufficiently large.  
\end{corollary}

The rest of this section is devoted to the proof of Theorem \ref{thm.ls} and Corollary \ref{cor.ls}.
First, in the next Section \ref{sec.sub.tfli} we prove several velocity weighted Lyapunov inequalities for the linear evolution \eqref{ls} which are pointwise in $(t,k)$.  Then, in Section \ref{sec.tf}, we will use these inequalities to prove the time decay.

\subsection{Weighted time-frequency Lyapunov inequalities}\label{sec.sub.tfli}
In this subsection, we shall construct the desired weighted time-frequency
Lyapunov functional as in Theorem \ref{thm.tfli}.  In the proof we have to take great care to estimate the microscopic and macroscopic parts for $|k| \le 1$ and $|k| > 1$ respectively each in different ways in order to capture the delicate individual behavior of each separate piece.

\subsubsection{Estimate on the microscopic dissipation}

The first step in our construction of the time-frequency Lyapunov
functional is to estimate the microscopic dissipation on the basis of the
coercivity property \eqref{coerc} of $\FL$. 

Consider  \eqref{ls}, taking the Fourier transform in $x$ grants us
\begin{equation}\label{ls-1f}
  \dis     \pa_t \hat{\solU}+\rmi\vel\cdot k ~ \hat{\solU}+\FL \hat{\solU} =\hat{\sourceG}.
\end{equation}
Then we multiply equation \eqref{ls-1f} with $\overline{\hat{\solU}(t,k,\vel)}$ and integrate over $\vel$ to achieve  
\begin{equation}
\notag
    \frac{1}{2}\frac{d}{dt} \nsm\hat{\solU}(t,k)\nsm_{L^2}^2+\rmre \ang{\FL \hat{\solU}, \hat{\solU} }
    =\rmre \ang{\hat{\sourceG}, \hat{\solU}}.
\end{equation}
From the coercivity estimate \eqref{coerc} and \eqref{form.p} one has that
\begin{equation}\label{diss-micr}
\frac{d}{dt} \nsm\hat{\solU}(t,k)\nsm_{L^2}^2
+
\la \nsm \{\FI-\FP\} \hat{\solU}\nsm_{L^2_{\gamma+2s}}^2
\lesssim
\left|  \rmre \ang{\hat{\sourceG}, \hat{\solU}} \right|.
\end{equation}
This is the first main estimate which we will use in the following.

Notice that in \eqref{diss-micr} we use the inclusion $L^2_{\gamma+2s} \supset \spacen$ from \eqref{normdef}.  This lower bound will also be implicitly used several times below since the $L^2_{\gamma+2s}(\threed_v)$ norm already captures the control that we will need in order to prove the linear decay.

\subsubsection{Microscopic weighted time-frequency  inequality}
In this section we prove the following instantaneous Lyapunov inequality
with a velocity weight $\wN \in \R$:
\begin{multline}\label{macroWeightINEQ}
\frac{d}{dt}\nsm w^{\wN}\{\FI-\FP\}\hat{\solU}(t,k)\nsm^2_{L^2} +
\la \nsm w^{\wN}\{\FI-\FP\}\hat{\solU}(t,k)\nsm^2_{L^2_{\gamma+2s}}
\\
\le C_\lambda |k|^2 \nsm \hat{\solU}\nsm^2_{L^2_{\gamma+2s}}
+
 C \nsm\{\FI-\FP\}\hat{\solU}\nsm_{L^2(B_{C})}^2
+
\left| \langle w^{2\wN} \{\FI-\FP\}\hat{\sourceG},\{\FI-\FP\} \hat{\solU}\rangle \right|.
\end{multline}
We split the solution $\solU$ to equation \eqref{ls} into $\solU=\FP \solU + \{\FI-\FP\}\solU$, take the Fourier transform as in \eqref{ls-1f},  and then apply $\{\FI-\FP\}$ to the resulting equation:
\begin{multline}\notag
\pa_t \{\FI-\FP\}\hat{\solU} + \rmi \vel \cdot k \{\FI-\FP\}\hat{\solU} 
+
\FL\{\FI-\FP\} \hat{\solU} 
=\{\FI-\FP\}\hat{\sourceG}
 \\
-\{\FI-\FP\}(\rmi \vel \cdot k \FP \hat{\solU} )
 +\FP (\rmi \vel \cdot k \{\FI-\FP\}\hat{\solU}).
\end{multline}
Multiply the last equation by $w^{2\wN} \overline{\{\FI-\FP\}\hat{\solU}}$ and integrate in $\vel$ to obtain
\begin{equation}
\label{app.vw.p05.int1.new}
\frac{1}{2}\frac{d}{dt}\nsm w^{\wN}\{\FI-\FP\}\hat{\solU}(t,k)\nsm^2_{L^2} +
 \rmre \langle w^{2\wN}\FL\{\FI-\FP\} \hat{\solU},\{\FI-\FP\} \hat{\solU}\rangle
 =
 \Ga_1+ \Ga_2,
\end{equation}
where $ \Ga_1 = \rmre \langle w^{2\wN} \{\FI-\FP\}\hat{\sourceG},\{\FI-\FP\} \hat{\solU}\rangle$ and
\begin{multline*}
 \Ga_2 =
-\rmre\left\langle
\{\FI-\FP\}(\rmi \vel \cdot k\FP \hat{\solU} ),
w^{2\wN}  \{\FI-\FP\}\hat{\solU}
\right\rangle
\\
 +\rmre \left\langle
 \FP (\rmi \vel \cdot k \{\FI-\FP\}\hat{\solU}),
w^{2\wN}  \{\FI-\FP\}\hat{\solU}  \right\rangle.
\end{multline*}
We will estimate each of the
three terms in \eqref{app.vw.p05.int1.new}.

As a result of the rapid decay in the coefficients of \eqref{form.p} we obtain
$$
\left| \Ga_2 \right| \le
\eta \nsm w^{\wN}\{\FI-\FP\}\hat{\solU}\nsm^2_{L^2_{\gamma+2s}}
+
C_\eta |k|^2 \left( \nsm w^{-\wE}\{\FI-\FP\}\hat{\solU}\nsm^2_{L^2}+\nsm \FP \hat{\solU}\nsm^2_{L^2} \right), 
$$
which holds for any small $\eta >0$ and any large $\wE>0$. For the second estimate, we invoke \cite[Lemma 2.6]{gsNonCut0} to achieve the following coercive bound
$$
\rmre  \langle w^{2\wN}\FL\{\FI-\FP\} \hat{\solU},\{\FI-\FP\} \hat{\solU}\rangle
 \ge \la \nsm w^{\wN}\{\FI-\FP\}\hat{\solU}\nsm^2_{L^2_{\gamma+2s}}
 -
 C  \nsm\{\FI-\FP\}\hat{\solU}\nsm_{L^2(B_{C})}^2.
$$
(Now (2.10) in \cite[Lemma 2.6]{gsNonCut0} indeed holds for any $\ell \in \R$ as follows from \cite[Lemma 2.4 and Lemma 2.5]{gsNonCut0}.)  Plug the last estimates into \eqref{app.vw.p05.int1.new} to obtain \eqref{macroWeightINEQ}.

Note that following the procedure as above when $\wN =0$, using \eqref{coerc}, yields
\begin{equation}\label{macroWeightINEQ.noL}
\frac{d}{dt}\nsm \{\FI-\FP\}\hat{\solU}(t,k)\nsm^2_{L^2} +
\la \nsm \{\FI-\FP\}\hat{\solU}(t,k)\nsm^2_{L^2_{\gamma+2s}}
\lesssim
 |k|^2 \nsm \FP \hat{\solU}\nsm^2_{L^2}.
\end{equation}
This holds when $\sourceG=0$, it will be useful in the proof of Corollary \ref{cor.ls}.

We furthermore remark, following the same procedure as above, that we get
\begin{equation}\label{macroWeightINEQ.noM}
\frac{1}{2}\frac{d}{dt}\nsm w^{\wN}\hat{\solU}(t,k)\nsm^2_{L^2} +
\la \nsm w^{\wN}\hat{\solU}(t,k)\nsm^2_{L^2_{\gamma+2s}}
\le 
 C  \nsm\hat{\solU}\nsm_{L^2(B_{C})}^2
+
\left| \langle w^{2\wN} \hat{\sourceG}, \hat{\solU}\rangle \right|.
\end{equation}
In other words, if we multiply \eqref{ls-1f} by $w^{2\wN}\overline{\hat{\solU}(t,k)}$, integrate in $\vel$ and use the same estimates as in the last case it follows that we obtain \eqref{macroWeightINEQ.noM}.

\subsubsection{Estimate on the macroscopic dissipation}\label{macro.diss.d}
In this section, we recall some arguments from \cite{D-Hypo} which are used to estimate the macroscopic dissipation, in the spirit of \cite{MR1057534}.   Now the form \eqref{form.p} of the orthogonal projection $\FP$ implies the identities
\begin{equation}
\label{coef.p.def}
\begin{split}
  \dis   & a= \langle \MM, \solU\rangle= \langle \MM, \FP \solU\rangle,
  \\
  \dis  & b_i=\langle \vel_i \MM, \solU\rangle
=\langle \vel_i \MM,\FP \solU\rangle,
\\
  \dis & c= \frac{1}{2\Ndim}\langle (|\vel|^2-\Ndim) \MM, \solU\rangle
= \frac{1}{2\Ndim}\langle (|\vel|^2-\Ndim) \MM, \FP \solU\rangle.
    \end{split}
\end{equation}
We will also use the following high-order moment functions $\highG(\solU)=(\highG_{ij}(\solU))_{\Ndim\times \Ndim}$ and 
$\highB(\solU)=(\highB_1(\solU),\cdots,\highB_\Ndim(\solU))$ which are given by
\begin{equation}
  \highG_{ij}(\solU) = \langle(\vel_i\vel_j-1)\MM, \solU\rangle,\ \
  \highB_i(\solU)=\langle(|\vel|^2-\Ndim - 2)\vel_i\MM,
  \solU\rangle.\label{def.gala}
\end{equation}
Now as in \cite{D-Hypo} these high-order moment functions satisfy some mixed hyperbolic-parabolic equations.  In the case when $\FP g=0$, the equations for these moment functions can be used to prove the following lemma from \cite[Lemma 4.1]{D-Hypo}:

\begin{lemma}
There is a time-frequency functional $\CE_{\rm int}(t,k)$ defined by 
\begin{eqnarray}
  \label{def.int1}
  \CE_{\rm int}(t,k) &=& \frac{1}{1+|k|^2}\sum_{i=1}^{\Ndim}\frac{1}{2} (\rmi k_i \hat{c}\mid \highB_i(\{\FI-\FP\}\hat{\solU}))\\
  \notag
  &&+\frac{\ka_1}{1+|k|^2}\sum_{i,j=1}^{\Ndim}(\rmi k_i \hat{b}_j+\rmi k_j\hat{b}_i\mid \frac{1}{2} \highG_{ij}(\{\FI-\FP\}\hat{\solU})+2\hat{c}\de_{ij})\\
  && +\frac{\ka_2}{1+|k|^2}\sum_{i=1}^{\Ndim}\left(\rmi k_i \hat{a}\mid \hat{b}_i\right),
\notag
\end{eqnarray}
with two properly chosen constants $0<\kappa_2\ll\kappa_1\ll 1$ such that
\begin{equation}
\dis \pa_t \rmre \CE_{\rm int}(t,k)+\frac{\la |k|^2}{1+|k|^2} \left(|\widehat{a}|^2+|\hat{b}|^2+|\hat{c}|^2\right)
\dis 
\lesssim
\left| \ang{\{\FI-\FP\}\hat{\solU}, \Rap} \right|^2
+
\left| \ang{\hat{\sourceG}, \Rap} \right|^2,
\label{diss-macro+}
\end{equation}
holds for any $t\geq 0$, and $k\in \threed$.  Now the $\{ \Rap_{\ArbI} \}$ are the smooth exponentially decaying velocity basis vectors which are contained in \eqref{form.p} and \eqref{def.gala}, and $\Rap$ is a linear combination of the $\{ \Rap_{\ArbI} \}$ whose precise form is unimportant herein.  
\end{lemma}

Note that in \eqref{diss-macro+} the upper bound is slightly different from what is written in the statement of \cite[Lemma 4.1]{D-Hypo}. However following the proof of \cite[Lemma 4.1]{D-Hypo} it is easy to see that our weaker upper bound indeed holds true.

\subsubsection{Derivation of the weighted time-frequency Lyapunov inequality}
In this subsection, we collect all of the disparate estimates from the previous subsections to prove in Theorem \ref{thm.tfli}
the main weighted time-frequency energy inequality.

\begin{theorem}\label{thm.tfli}
Fix $\wN\in \R$.
Let $\solU$ be the solution to the Cauchy problem \eqref{ls}
with $g=0$. Then there is a weighted time-frequency functional $\CE_\wN(t,k)$
such that
\begin{equation}\label{thm.tfli.1}
    \CE_\wN(t,k) \approx \nsm w^\ell \hat{\solU}(t,k)\nsm_{L^2}^2,
\end{equation}
where for any $t\geq 0$ and $k\in \threed$ we have 
\begin{equation}\label{thm.tfli.2}
\dis\pa_t \CE_{\wN} (t,k)+\la \left( 1 \wedge |k|^2\right) \nsm w^{\wN}  \hat{\solU}(t,k)\nsm_{L^2_{\gamma+2s}}^2
\le 0.
\end{equation}
We use the notation $1 \wedge |k|^2 \eqdef\min\{1,|k|^2\}$.
\end{theorem}

\begin{proof}
We first define 
\begin{equation}\label{thm.tfli.p1}
    \CE(t,k)\eqdef \nsm\hat{\solU}(t,k)\nsm_{L^2}^2
    +\kappa_3 \CE_{\rm int}(t,k),
\end{equation}
for a constant $\kappa_3>0$ to be determined later, where $\CE_{\rm int}(t,k)$ is given in \eqref{def.int1}. 
One can fix $\kappa_3>0$ small enough such that 
$
\CE(t,k)\approx \nsm\hat{\solU}(t,k)\nsm_{L^2}^2.
$

A linear combination of \eqref{diss-micr} and \eqref{diss-macro+} according to \eqref{thm.tfli.p1}  implies  that
\begin{equation}\label{instant.noW}
\dis\pa_t \CE(t,k)+\la \nsm\{\FI-\FP\}\hat{\solU}\nsm_{L^2_{\gamma+2s}}^2
+\frac{\la |k|^2}{1+|k|^2}(|\hat{a}|^2
+|\hat{b}|^2+|\hat{c}|^2)
\le 0,
\end{equation}
where note further that 
$
|\hat{a}|^2+|\hat{b}|^2+|\hat{c}|^2
\gtrsim
\nsm
\FP \hat{\solU}\nsm_{L^2_\wN}^2 
$
holding $\forall \wN\in \R$.

To do the weighted estimates, in particular for the soft potentials \eqref{kernelPsing}, we need to introduce a new energy splitting as follows.
With \eqref{thm.tfli.p1} we define 
\begin{equation}\notag
\begin{split}
    \CE_\wN^0(t,k)\eqdef &
    \ind_{|k|\le 1} \left(
    \CE(t,k)
    +\kappa_4  \nsm w^{\wN}\{\FI-\FP\}\hat{\solU}(t,k)\nsm_{L^2}^2 \right),
    \\
        \CE_\wN^1(t,k)\eqdef &
    \ind_{|k|> 1} \left(
    \CE(t,k)
    +\kappa_5  \nsm w^{\wN}\hat{\solU}(t,k)\nsm_{L^2}^2 \right).
    \end{split}
\end{equation}
Again $ \kappa_4, \kappa_5>0$ will be determined later.

We prove estimates for each of these individually.  For $\CE^1_{\wN}(t,k)$ we combine \eqref{instant.noW} with \eqref{macroWeightINEQ.noM} for $|k|> 1$ to  obtain for a suitably small $\kappa_5>0$  that
$$
\partial_t  \mathcal{E}^1_{\wN}(t,k) 
+ \la  \nsm w^\wN  \hat{\solU}(t,k)\nsm_{L^2_{\gamma + 2s}}^2 \ind_{|k|> 1} 
\le 0.
$$
Here we have used the fact that when $|k|> 1$ then 
$
\frac{ |k|^2}{1+|k|^2} \ge \frac{1}{2}.
$

Furthermore, when $|k|\le 1$ it holds that
$
\frac{ |k|^2}{1+|k|^2} \ge \frac{|k|^2}{2}.
$
In this case we combine \eqref{instant.noW} with \eqref{macroWeightINEQ}  on $|k|\le 1$ to  obtain for a small $\kappa_4>0$  that
\begin{equation}\notag
\partial_t  \mathcal{E}_{\wN}^0(t,k) 
+ \la |k|^2 \nsm w^\wN  \hat{\solU}(t,k)\nsm_{L^2_{\gamma + 2s}}^2  \ind_{|k|\le 1} 
\le 0.
\end{equation}
Lastly we define 
$
\mathcal{E}_{\wN}(t,k) 
\eqdef
\mathcal{E}_{\wN}^0(t,k) 
+
\mathcal{E}_{\wN}^1(t,k) 
$
and we notice that \eqref{thm.tfli.1} is satisfied.
Then \eqref{thm.tfli.2} follows from adding the previous two differential inequalities.
\end{proof}

\subsection{Proof of time decay of linear solutions}\label{sec.tf}
In this subsection we prove Theorem \ref{thm.ls} based on Theorem \ref{thm.tfli}.
For the soft-potentials \eqref{kernelPsing}, the time-frequency dissipation is weaker than $\CE_\wN(t,k)$ so we close the estimates using interpolation.

\begin{proof}[Proof of Theorem \ref{thm.ls}]  
Firstly we restrict to the hard potentials \eqref{kernelP}.
We define 
$
  \fcnP(k)\eqdef
    \la \left( 1 \wedge |k|^2\right).
$
Since $\gamma + 2s \ge 0$ and $\sourceG=0$, 
we can rewrite \eqref{thm.tfli.2}, for any $t\geq 0$ and $k\in \threed$, as
\begin{equation*}
    \pa_t \CE_\wN(t,k)+\fcnP(k)  \CE_\wN(t,k)
\le 0.
\end{equation*}
The Gronwall inequality then gives us that
\begin{equation}\label{thm.ls.p02}
   \CE_\wN(t,k)\leq e^{-\fcnP(k)t}\CE_\wN(0,k).
\end{equation}
As a result of \eqref{thm.tfli.1}, for any fixed $m\geq 0$ and $\wN \in \R$ we notice that
\begin{equation}\label{thm.ls.p08}
    \|w^{\wN} \solU(t)\|_{\CHd^m}^2 \approx \int_{\threed} |k|^{2m}\CE_\wN(t,k)dk.
\end{equation}
Now, to prove \eqref{thm.ls.1}, we can apply \eqref{thm.ls.p02}.  
Here, notice that for $|k|\leq 1$,
$
\fcnP(k)\geq \la|k|^2,
$
and for $|k|\geq 1$,
$
\fcnP(k)\geq \la.
$
With these two estimates we have the upper bound
$$
 \int_{\threed}dk~ |k|^{2m}\CE_\wN(t,k) \leq   
 \int_{|k|\leq 1}dk~ |k|^{2m}e^{-\la |k|^2t}\CE_\wN(0,k)
 +e^{-\la t} \int_{|k|\geq 1}dk~ |k|^{2m}\CE_\wN(0,k).
$$
In \eqref{ineqW} from Appendix \ref{secAPP:HHY}, using the H\"{o}lder and Hausdorff-Young inequalities, we showed that the integration over $|k|\leq 1$ is bounded for $1\leq r \leq 2$ as follows
$$
 \int_{|k|\leq 1}|k|^{2m}e^{-\la |k|^2t}\CE_\wN(0,k)dk
\lesssim (1+t)^{-2\sigma_{r,m}}
\|w^{\wN}\solU_0\|_{Z_r}^2.
$$
Here we recall \eqref{rateLIN}.  For the integration over $|k|\geq 1$:
\begin{equation}
\notag
 \int_{|k|\geq 1}|k|^{2m}e^{-\la t} \CE_\wN(0,k)dk
 \lesssim
 e^{-\la t} \|w^\wN\solU_0\|_{\CHd^m}^2.
\end{equation}
Collecting the above estimates as well as \eqref{thm.ls.p08} gives \eqref{thm.ls.1}.

It remains to prove \eqref{thm.ls.1.soft} for the soft potentials \eqref{kernelPsing}.  
By \eqref{thm.tfli.2} we have that
$
\mathcal{E}_{\wN}(t,k)
\le
\mathcal{E}_{\wN}(0,k)
$
for any $\wN \in \R$.
On the other hand notice that \eqref{thm.tfli.2} is insufficient because $\gamma + 2s < 0$, and in this case $\mathcal{E}_{\wN}(t,k)$ is not controlled by $\nsm w^\wN  \hat{\solU}(t,k)\nsm_{L^2_{\gamma+2s}}^2$.
To resolve this difficulty we interpolate with a family of norms.  

In particular, for $\wE>0$, using \eqref{thm.tfli.1} and \eqref{weigh} we have
$$
\mathcal{E}_{\wN}(t,k)
\lesssim
 \mathcal{E}_{\wN-1}^{\wE/(\wE+1)}(t,k) ~ \mathcal{E}_{\wN+\wE}^{1/(\wE+1)}(t,k)
\lesssim
\nsm w^\wN  \hat{\solU}(t,k)\nsm_{L^2_{\gamma+2s}}^{2 \wE/(\wE+1)}\mathcal{E}_{\wN+\wE}^{1/(\wE+1)}(t,k).
$$
We therefore conclude that
$$
\mathcal{E}_{\wN}^{(\wE+1)/\wE}(t,k)
\lesssim
\nsm w^\wN  \hat{\solU}(t,k)\nsm_{L^2_{\gamma+2s}}^{2 }\mathcal{E}_{\wN+\wE}^{1/\wE}(t,k)
\lesssim
\nsm w^\wN  \hat{\solU}(t,k)\nsm_{L^2_{\gamma+2s}}^{2 }\mathcal{E}_{\wN+\wE}^{1/\wE}(0,k).
$$
Now we can rewrite \eqref{thm.tfli.2}, for any  $k\in \threed$, as
\begin{equation*}
    \pa_t \CE_{\wN}(t,k)+ \la \fcnP(k)  \mathcal{E}_{\wN}^{(\wE+1)/\wE}(t,k)\mathcal{E}_{\wN+\wE}^{-1/\wE}(0,k)  \leq 0.
\end{equation*}
To prove \eqref{thm.ls.1.soft}, one can bound $\CE_{\wN}(t,k)$ as follows
\begin{equation*}
    \pa_t \CE_{\wN}(t,k)  \mathcal{E}_{\wN}^{-1-1/\wE}(t,k)    \lesssim -\fcnP(k)  \mathcal{E}_{\wN+\wE}^{-1/\wE}(0,k).
\end{equation*}
Integrating this over time, we obtain  
\begin{equation*}
    \wE \CE_{\wN}^{-1/\wE}(0,k) - \wE \mathcal{E}_{\wN}^{-1/\wE}(t,k)    \lesssim - t \fcnP(k)  \mathcal{E}_{\wN+\wE}^{-1/\wE}(0,k).
\end{equation*}
For any $\wN \in \R$ and $\wE>0$, uniformly in $k\in \threed$, we have shown that
\begin{equation*}
     \mathcal{E}_\wN(t,k)    \lesssim
    \mathcal{E}_{\wN+\wE}(0,k) 
    \left(\frac{ t \fcnP(k)}{\wE}   + 1\right)^{-\wE}.
\end{equation*}
We also just used the estimate $\mathcal{E}_{\wN}(0,k) \lesssim\mathcal{E}_{\wN+\wE}(0,k)$.

As before, we integrate over $k$ and split into $|k|\leq 1$ and $|k|> 1$ to achieve
\begin{equation*}
\int_{|k| > 1}dk~ |k|^{2m}
     \mathcal{E}_\wN(t,k)   \lesssim
     \left(\frac{ t }{\wE}   + 1\right)^{-\wE}
  \int_{|k|> 1}dk~ |k|^{2m}  \mathcal{E}_{\wN+\wE}(0,k).
\end{equation*}
Alternatively, when $|k|\leq 1$ we choose $\wE$ to satisfy $\wE > 2 \sigma_{r,m}$ and obtain  
\begin{multline*}
\int_{|k| \leq 1}dk~ |k|^{2m}
     \mathcal{E}_\wN(t,k)    
\lesssim     
     \int_{|k| \leq 1}dk~ |k|^{2m}
    \mathcal{E}_{\wN+\wE}(0,k) 
    \left(\frac{ t |k|^2}{\wE}   + 1\right)^{-\wE}
    \\
    \lesssim     
    \left( 1   + t\right)^{-2\sigma_{r,m}}
     \| w^{\wN+\wE} \solU_0\|_{Z_r}^2.
\end{multline*}
For $1\le r \le 2$, this last inequality again uses \eqref{ineqW} in Appendix \ref{secAPP:HHY}.
\end{proof}

Next we give our proof of Corollary \ref{cor.ls}, using different methods.

\begin{proof}[Proof of Corollary \ref{cor.ls}]  To prove time decay, now and in subsequent sections
without loss of generality we can suppose that $t\ge 1$.  
We will use a time-velocity splitting on the energy inequality from \eqref{macroWeightINEQ.noL}.  Recalling \eqref{weigh}, we define the sets
\begin{equation}
E(t) = \{w(v)\le  t^{p^\prime}\}, \quad  E^c(t)= \{ w(v) >  t^{p^\prime}\}.
\label{split.E}
\end{equation}
Here $p^\prime\ge 0$ will be chosen later. We initially restrict to the case of the soft potentials \eqref{kernelPsing}.  
For the dissipation term in the energy inequality \eqref{macroWeightINEQ.noL} we have
\begin{equation}
\frac{\nsm \ind_E \{\FI-\FP\}\hat{\solU}(t,k)\nsm^2_{L^2}}{ t^{p^\prime} }
\lesssim 
\nsm \ind_E \{\FI-\FP\}\hat{\solU}(t,k)\nsm^2_{L^2_{\gamma+2s}}.
\label{soft.up.d.lin}
\end{equation}
We plug \eqref{soft.up.d.lin} into \eqref{macroWeightINEQ.noL} to achieve
\begin{equation}\notag
\frac{d}{dt}\nsm \{\FI-\FP\}\hat{\solU}(t,k)\nsm^2_{L^2} 
+
\frac{\la^\prime}{t^{p^\prime} } \nsm \{\FI-\FP\}\hat{\solU}\nsm^2_{L^2}
\lesssim
 |k|^2 \nsm \FP \hat{\solU}\nsm^2_{L^2}
 +
 \frac{\la^\prime}{t^{p^\prime} } \nsm \ind_{E^c} \{\FI-\FP\}\hat{\solU}\nsm^2_{L^2}.
\end{equation}
Define $\lambda= \frac{\la^\prime}{ p } $ where now $p= -p^\prime+1>0$.  Then we have
\begin{equation*}
\frac{d}{dt}\nsm \{\FI-\FP\}\hat{\solU}(t,k)\nsm^2_{L^2} 
+
\lambda p t^{p-1} \nsm \{\FI-\FP\}\hat{\solU}\nsm^2_{L^2} 
\lesssim
 |k|^2 \nsm \FP \hat{\solU}\nsm^2_{L^2}
 +
t^{p-1} \nsm \ind_{E^c} \{\FI-\FP\}\hat{\solU}\nsm^2_{L^2}.
\end{equation*}
Equivalently 
\begin{equation*}
\frac{d}{dt}\left(e^{\lambda t^{p}} \nsm \{\FI-\FP\}\hat{\solU}(t,k)\nsm^2_{L^2}   \right) 
\lesssim
e^{\lambda t^p} |k|^2 \nsm \FP \hat{\solU}\nsm^2_{L^2}
+
 t^{p-1} e^{\lambda t^p}
\nsm \ind_{E^c} \{\FI-\FP\}\hat{\solU}\nsm^2_{L^2}.
\end{equation*}
The integrated form of this inequality is 
\begin{multline}
\nsm \{\FI-\FP\}\hat{\solU}(t,k)\nsm^2_{L^2}
\lesssim
e^{-\lambda t^{p}}\nsm \{\FI-\FP\}\hat{\solU}_0(k)\nsm^2_{L^2} 
\\
+   
\int_0^t 
ds ~ e^{-\lambda (t^p-s^p)}
\left( 
|k|^2 \nsm \FP \hat{\solU}\nsm^2_{L^2}
+
\Ac s^{p-1} \nsm \ind_{E^c} \{\FI-\FP\}\hat{\solU}(s,k)\nsm^2_{L^2}
\right).
\label{main.gs.split.lin}
\end{multline}
Here $\Ac \ge 0$, and we take $p>0$, or $0< p^\prime<1$,  so that the integral is finite.

In this case, 
since $\nsm \ind_{E^c} \{\FI-\FP\}\hat{\solU}(s,k)\nsm^2_{L^2}$ is restricted to $E^c(s)$  we have that
\begin{equation}
\label{decay.high.lin}
\nsm \ind_{E^c} \{\FI-\FP\}\hat{\solU}(s,k)\nsm^2_{L^2}
\lesssim
(1+s)^{-2\sigma_{r,m+1}} \CE_{\wE}(0,k),
\end{equation}
where we have used 
$
1  \lesssim w^{\frac{2}{p^\prime} \sigma_{r,m+1} }(v) (1+s)^{-2\sigma_{r,m+1}}
$
on $E^c(s)$ and \eqref{thm.tfli.2}.  

For the soft potentials \eqref{kernelPsing} the proof of Corollary \ref{cor.ls} now follows by multiplying \eqref{main.gs.split.lin} by $|k|^{2m}$, integrating over $\threed_k$, applying \eqref{thm.ls.1.soft} in  Theorem \ref{thm.ls}, and choosing $p^\prime>0$ small enough so that $\epsilon = \epsilon(p^\prime$) is arbitrarily small.  The hard potential \eqref{kernelP} case is similar since then we obtain \eqref{main.gs.split.lin} with $p=1$ and $\Ac=0$.  \end{proof}

This completes our estimates for the time decay of the Cauchy problem for the linearized non cut-off Boltzmann equation \eqref{ls}.  In the next section we explain how to prove these decay rates in the nonlinear regime, for \eqref{Boltz}.

\section{Nonlinear time decay}\label{sec.decayNL}

This section is devoted to the proof of Theorem \ref{thm.ns} and
Corollary \ref{cor}.   But first we prove the energy estimates \eqref{thm.energy.1} and \eqref{thm.energy.2} from Theorem \ref{thm.energy}.  Note that once these energy estimates are established, the rest of the statements in Theorem \ref{thm.energy} can be shown by using the methods elaborated in \cite{gsNonCut0}.

\subsection{Velocity-weighted energy estimates}\label{sec:vwee}
In this subsection, we will prove the velocity-weighted energy estimates in \eqref{thm.energy.1} for solutions to \eqref{Boltz}. The first step is to prove the following unweighted estimate for the norms in \eqref{def.eNm} and \eqref{def.dNm}:
\begin{multline}  
\label{zero.est}
\frac{d}{dt}\sum_{|\al|\leq \dK}C_\al \|\pa^\al \solU(t)\|_{\CL^2}^2
+
\la \noCD_\dK(t)
\\ 
\lesssim 
\sum_{|\al| \leq \dK} \left| \left( \pa^\al \Ga(f,f), \pa^\al f \right) \right|
\lesssim 
\sqrt{\CE_\dK(t)} \CD_\dK(t),
\end{multline}  
where $C_\alpha >0$, $\sum_{|\al|\leq \dK}C_\al \|\pa^\al \solU(t)\|_{\CL^2}^2$ denotes a continuous functional which is comparable to $\sum_{|\al|\leq \dK} \|\pa^\al \solU(t)\|_{\CL^2}^2$, and $\noCD_\dK(t)$ contains no velocity derivatives:
\begin{equation}
\label{dk.no.be}
\noCD_\dK(t)
\eqdef
 \sum_{1\le |\al| \leq \dK}   \| \partial^\alpha \solU(t) \|_{_{N^{s,\gamma}}}^2 
+
 \| \{ {\bf I - P } \} \solU(t) \|_{_{N^{s,\gamma}}}^2. 
\end{equation}
Following the proof of \cite[Theorem 8.4]{gsNonCut0}, we find for suitable solutions to \eqref{Boltz} that there are constants $\delta>0$ and $C_2>0$ such that
\begin{equation}
\sum_{|\alpha| \le \dK} \| \{ {\bf I - P } \} \partial^\alpha \solU(t) \|_{_{N^{s,\gamma}}}^2
\ge 
\delta
\sum_{|\alpha| \le \dK-1} \| \nabla_x{ \bf  P  } \partial^\alpha \solU(t) \|_{_{N^{s,\gamma}}}^2 - C_2\frac{d\mathcal{I}(t)}{dt}.
\label{coerc.m}
\end{equation}
Note that in the whole space we can not use the Poincar{\'e} inequality to obtain the term ${ \bf  P  }f$ in the lower bound.  
Otherwise the rest of the arguments in the proof of \cite[Theorem 8.4]{gsNonCut0} generalize from $\domain_x$ to $\threed_x$.  Furthermore $\mathcal{I}(t)$ is a suitable functional defined precisely in \cite[(8.25)]{gsNonCut0}.  The key property of $\mathcal{I}(t)$ is that it can be absorbed into the energy, for a small $\kappa >0$, as follows
\begin{equation}
 \|\solU(t)\|_{\CH^\dK}^2+\ka~  \mathcal{I}(t) \approx \sum_{|\al|\leq \dK}C_\al \|\pa^\al \solU(t)\|_{\CL^2}^2,
 \quad C_\al > 0.
\label{interact.f}
\end{equation}
Then with \eqref{coerc.m} and \eqref{interact.f} we obtain \eqref{zero.est} with the first upper bound. The proof of this inequality follows exactly the proof in \cite[(8.26)]{gsNonCut0} when $\ell = |\beta| =0$.

To obtain the second upper bound in \eqref{zero.est}, we need to work a bit harder because $\CD_{\dK, \wN}(t)$ in \eqref{def.dNm} does not contain $\FP \solU$.  To overcome this, we {\it claim} that
\begin{equation}
\label{decay.ga}
 \left| \left(  w^{2\wN-2|\beta|}\partial^\alpha_\beta \Ga(\solU,\solU), \partial^\alpha_\beta \{\FI-\FP\}\solU\right) \right| 
 \lesssim 
\sqrt{\CE_{\dK,\wN}(t)} \CD_{\dK,\wN}(t),
\end{equation}
where $\ell \ge 0$ and $|\al| + |\be| \le \dK$ with $\NgE$.  This clearly implies the second upper bound inequality in \eqref{zero.est} since
$
\left(  \Ga(f,f),  f \right)
=
\left(  \Ga(f,f),  \{\FI-\FP\}\solU \right).
$

The first step in our proof of \eqref{decay.ga} is to use the well-known expansion
\begin{equation}
\label{ga.expand}
\Ga(\solU,\solU) =
\Gamma (\FP \solU, \FP \solU)
+
\Ga (\{\FI-\FP\}\solU, \FP \solU)
+
\Ga (\solU, \{\FI-\FP\}\solU).
\end{equation}
The estimate \eqref{decay.ga} for the third term above, 
$
\Ga (\solU, \{\FI-\FP\}\solU),
$
follows directly from \cite[Lemmas 2.2 and 2.3]{gsNonCut0} since 
$\{\FI-\FP\}\solU$ is a part of $\CD_{\dK,\wN}(t)$ from \eqref{def.dNm}.

For the first and second terms in \eqref{ga.expand}, we notice from \eqref{form.p} that
$$
\Ga (\solU, \FP \solU)= \sum_{i=1}^{\Ndim +2} \psi_i(t,x) \Ga (\solU, \phi_i),
$$
where the $\psi_i(t,x)$ are the elements from \eqref{coef.p.def} and the $\phi_i(v)$ are the smooth rapidly decaying velocity basis vectors in \eqref{nullLsp}.  Thus from \cite[Proposition 6.1]{gsNonCut0}:
\begin{multline}\notag
\left| \left(  w^{2\wN-2|\beta|}\partial^\alpha_\beta \Gamma (\{\FI-\FP\} \solU, \FP \solU), \partial^\alpha_\beta \{\FI-\FP\}\solU \right) \right|
\\
\lesssim
\int_{\threed}dx ~ 
\nsm w^{\wN-|\beta|} \partial^\alpha_\beta \{\FI-\FP\}\solU\nsm_{L^2_{\gamma+2s}}
\sum_{\substack{\al_1 \le \al \\ \be_1 \le \be}}
|\pa^{\al_1} [a,b,c]| ~
  \nsm w^{\wN-|\beta|} \partial^{\al-{\al_1}}_{\be_1} \{\FI-\FP\} \solU\nsm_{L^2_{\gamma+2s}}.
\end{multline}
 Here $|[a,b,c]|$ is just the Euclidean square norm of the coefficients from \eqref{coef.p.def}.  Now take the supremum of either $|\pa^{\al_1} [a,b,c]|$ or $\nsm w^{\wN-|\beta|} \partial^{\al-{\al_1}}_{\be_1} \{\FI-\FP\} \solU\nsm_{L^2_{\gamma+2s}}$, whichever contains the largest total number of derivatives, and use the embedding $H^{\ksob}_x \subset L^\infty_x$ for this term and Cauchy-Schwartz for the others to obtain \eqref{decay.ga}.
 
 The last case to consider is $\Gamma (\FP \solU, \FP \solU)$.  From \cite[Proposition 6.1]{gsNonCut0} again
 \begin{multline}\notag
\left| \left(  w^{2\wN-2|\beta|}\partial^\alpha_\beta \Gamma (\FP \solU, \FP \solU), \partial^\alpha_\beta \{\FI-\FP\}\solU \right) \right|
\\
\lesssim
\int_{\threed}dx ~ 
\nsm w^{\wN-|\beta|} \partial^\alpha_\beta \{\FI-\FP\}\solU\nsm_{L^2_{\gamma+2s}}
\sum_{\al_1 \le \al }
|\pa^{\al_1} [a,b,c]| ~|\pa^{\al-\al_1} [a,b,c]|.
\end{multline}
This follows from the rapid decay of the basis vectors in \eqref{nullLsp}.  Now if $|\al|>0$ then we again use
$H^{\ksob}_x \subset L^\infty_x$ and Cauchy-Schwartz to get \eqref{decay.ga}.

However, when $\al =0$, we combine the $L^{2^*}(\threed_x)$ gradient Sobolev inequality, for $\Ndim \ge 3$
and $2^* = \frac{2n}{n-2}$,
with $W^{m,2^*}(\threed_x)\subset L^\infty (\threed_x)$, for $m> \frac{\Ndim}{2}-1$,
 to obtain
\begin{equation}
 \|\solU \|^2_{L^\infty(\threed_x)}
\lesssim
 \|\solU \|^2_{W^{m,2^*}(\threed_x)}
\lesssim
 \|\nabla_x \solU\|^2_{H^{m}(\threed_x)}.
\label{sobolev.trick}
\end{equation}
We can fortunately choose $m+1= \ksob$.    Now take the $L^\infty_x$ norm of 
$| [a,b,c]|$, use \eqref{sobolev.trick} and Cauchy-Schwartz to get \eqref{decay.ga} when $\al =0$ and $\Ndim \ge 3$.
Alternatively, when $\al =0$ and $\Ndim =2$ we use Cauchy-Schwartz combined with the following inequality
$$
\| [a,b,c] \|_{L^4(\R^2_x)}^2
\lesssim
\| [a,b,c] \|_{L^2(\R^2_x)}
\| \nabla_x [a,b,c] \|_{L^2(\R^2_x)}.
$$
We have then shown \eqref{decay.ga} in all the different cases.

In the rest of this subsection we will prove weighted energy estimates with spatial derivatives for the macroscopic part, $\{\FI-\FP\}\solU$ in Step 1.  Then in Step 2 we will prove velocity weighted estimates with both space and velocity derivatives also for $\{\FI-\FP\}\solU$.  A suitable linear combination of these and \eqref{zero.est} will establish \eqref{thm.energy.1}.

\medskip

\noindent{\it Step 1.} We split the solution $\solU$ to equation \eqref{Boltz} into $\solU=\FP \solU + \{\FI-\FP\}\solU$ and take
$\{\FI-\FP\}$ of the resulting equation to obtain
\begin{multline}\label{app.vw.p04}
\pa_t \{\FI-\FP\}\solU + \vel \cdot \na_x \{\FI-\FP\}\solU 
+
\FL\{\FI-\FP\} \solU 
=
\Ga(\solU,\solU)
 \\
-\{\FI-\FP\}(\vel \cdot \na_x \FP \solU )
 +\FP (\vel \cdot \na_x \{\FI-\FP\}\solU).
\end{multline}
For $|\al|\le \dK-1$ we take $\pa^\al$ of \eqref{app.vw.p04}, multiply the result by $w^{2\wN}\pa^\al\{\FI-\FP\}\solU$ for $\wN \ge 0$, and then integrate in $x,\vel$ to achieve:
\begin{equation}
\label{app.vw.p05.int1}
\frac{1}{2}\frac{d}{dt}\|w^{\wN}\pa^\al\{\FI-\FP\}\solU\|_{\CL^2}^2 +
 \left( w^{2\wN}\FL\pa^\al\{\FI-\FP\}\solU,\pa^\al\{\FI-\FP\} \solU\right)
 =
 \Ga_1+ \Ga_2,
\end{equation}
where $ \Ga_1 =  \langle w^{2\wN} \pa^\al\Ga(\solU,\solU),\pa^\al\{\FI-\FP\} \solU\rangle$,
and
\begin{multline*}
 \Ga_2 =
-\left(
\{\FI-\FP\}(\vel \cdot \na_x \FP\pa^\al \solU ),
w^{2\wN} \{\FI-\FP\}\pa^\al\solU
\right)
\\
 +\left(
 \FP (\vel \cdot \na_x \{\FI-\FP\}\pa^\al\solU),
w^{2\wN}  \{\FI-\FP\}\pa^\al\solU  \right).
\end{multline*}
We will estimate each of the three terms in \eqref{app.vw.p05.int1}.

Now from \eqref{decay.ga} we have 
$
\left| \Ga_1 \right| 
\lesssim
\sqrt{\CE_{\dK,\wN}(t)}\CD_{\dK,\wN}(t).
$
Then with Cauchy-Schwartz, we have
$
\left| \Ga_2 \right|
\lesssim
\noCD_\dK(t)
$
from \eqref{dk.no.be}.
Furthermore, from \cite[Lemma 2.6]{gsNonCut0}, we see that
$$
 \left( w^{2\wN}\FL\{\FI-\FP\}\pa^\al \solU,\{\FI-\FP\} \pa^\al\solU\right)
 \ge \la \|  \{\FI-\FP\}\pa^\al\solU\|_{\spacen_{\wN}}^2
 -
 C_\lambda  \|\{\FI-\FP\}\pa^\al\solU\|_{L^2_{\gamma+2s}}^2.
$$
We collect these estimates into \eqref{app.vw.p05.int1} to achieve the final estimate of
\begin{eqnarray}
\frac{1}{2}\frac{d}{dt}\sum_{|\al| \le \dK-1} \|w^{\wN}\{\FI-\FP\}\pa^\al\solU\|_{\CL^2}^2
&+&
\la \sum_{|\al| \le \dK-1}\|\{\FI-\FP\}\pa^\al\solU\|_{\spacen_{\wN}}^2\label{app.vw.p05}
\\
&&\lesssim 
\noCD_\dK(t)+ \sqrt{\CE_{\dK,\wN}(t)}\CD_{\dK,\wN}(t).
\notag
\end{eqnarray}
Similarly, for $|\al | =\dK$ we take $\pa^\al$ of \eqref{Boltz}, multiply the result by $w^{2\wN}\pa^\al \solU$ for $\wN \ge 0$, and then integrate in $x,\vel$ and sum over $|\al | =\dK$ to achieve:
\begin{equation}
\sum_{|\al| = \dK} \left( \frac{1}{2} \frac{d}{dt} \|w^{\wN}\pa^\al\solU\|_{\CL^2}^2
+
\la \|\pa^\al\solU\|_{\spacen_{\wN}}^2\right)\label{app.vw.pp05}
\lesssim 
\noCD_\dK(t)+ \sqrt{\CE_{\dK,\wN}(t)}\CD_{\dK,\wN}(t).
\end{equation}
These are the main energy inequalities in the first step in our proof of \eqref{thm.energy.1}.

\medskip

\noindent{\it Step 2.}  Fix $|\al|+|\be|\leq \dK$ with $|\be|\geq 1$. We apply $\pa^\al_\be$ to \eqref{app.vw.p04} to obtain
\begin{equation}
\pa_t \pa^\al_\be \{\FI-\FP\} \solU +\vel\cdot \na_x \pa^\al_\be \{\FI-\FP\} \solU + 
\pa_\be\FL \pa^\al \{\FI-\FP\} \solU =\pa_\be^\al \Ga(\solU,\solU)+I_1+I_2,\label{app.vw.p07}
\end{equation}
where the functionals $I_1$ and $I_2$ are denoted by
$$
I_1\eqdef-\pa_\be^\al\{\FI-\FP\}(\vel \cdot \na_x \FP \solU )
+\pa_\be^\al\FP (\vel \cdot \na_x \{\FI-\FP\}\solU ),
$$
and for $C_{\be_1}^\be \ge 0$ we have 
$$
  I_2 \eqdef -\sum_{|\beta_1|=1} C_{\be_1}^\be (\pa_{\be_1}\vel)\cdot \na_x  \pa^\al_{\be-\be_1} \{\FI-\FP\} \solU.
$$
We multiply \eqref{app.vw.p07} by $w^{2\wN-2|\beta|} \pa^\al_\be \{\FI-\FP\} \solU$ and integrate over $x,\vel$ to get
$$
\frac{1}{2}\frac{d}{dt}\|w^{\wN-|\beta|}\pa^\al_\be \{\FI-\FP\} \solU\|^2_{\CL^2}
+
\left(w^{2\wN-2|\beta|}\pa_\be\FL \pa^\al \{\FI-\FP\} \solU,\pa^\al_\be \{\FI-\FP\} \solU\right)
=
\sum_{j=1}^3\tilde{I}_j.
$$
Above $\tilde{I}_j \eqdef \left( I_j, w^{2\wN-2|\beta|} \pa^\al_\be \{\FI-\FP\} \solU\right)$ for $j=1,2$ and
$$
\tilde{I}_3 \eqdef  
\left(w^{2\wN-2|\beta|} \pa_\be^\al \Ga(\solU,\solU),\pa^\al_\be \{\FI-\FP\} \solU\right).
$$
From \eqref{decay.ga} we have 
$
\left| \tilde{I}_3 \right| 
\lesssim
\sqrt{\CE_{\dK,\wN}(t)}\CD_{\dK,\wN}(t).
$
Analogous to the estimates in Step 1, 
$
\left| \tilde{I}_1 \right|
\lesssim
\noCD_{\dK}(t)
$
for \eqref{dk.no.be}.
Furthermore from \cite[Lemma 2.6]{gsNonCut0} for a small $\eta>0$
\begin{multline}\notag
\left(w^{2\wN-2|\beta|}\pa_\be\FL (\pa^\al \{\FI-\FP\} \solU),\pa^\al_\be \{\FI-\FP\} \solU\right)
\gtrsim
\| \partial^{\alpha}_{\beta}\{\FI-\FP\} \solU\|_{N^{s,\gamma}_{\ell - |\beta|}}^2   
\\  
-
\eta
\sum_{|\beta_1| \le |\beta|}\| \partial^{\alpha}_{\beta_1} \{\FI-\FP\} \solU\|_{N^{s,\gamma}_{\ell - |\beta_1|}}^2  
-
C     \| \partial^\alpha \{\FI-\FP\} \solU\|_{L^2(B_{C})}^2.
\end{multline}
Lastly considering $\tilde{I}_2$ as in \cite[(8.29)]{gsNonCut0} for any small $\eta'>0$ we see that
$$
\left| \tilde{I}_2 \right| 
\le
\eta'
\|\partial_{\beta}^{\alpha}\{\FI-\FP\}\solU(t)\|_{N^{s,\gamma}_{\ell - |\beta|}}^{2}
+
C_{\eta'}\|\nabla_{x}\partial_{\beta-\beta_{1}}^{\alpha}\{\FI-\FP\}\solU \|_{N^{s,\gamma}_{\ell - |\beta-\beta_1|}}^{2}.
$$
We add together each of these estimates, use a simple induction, and sum to obtain
\begin{multline}
\frac{1}{2}\frac{d}{dt}\sum_{\substack{|\al|+|\be|\leq \dK\\ |\be|\geq 1}}\|w^{\wN-|\be|}\pa^\al_\be \{\FI-\FP\} \solU\|_{\CL^2}^2
+
\la 
\sum_{\substack{|\al|+|\be|\leq \dK\\ |\be|\geq 1}}\|\pa^\al_\be \{\FI-\FP\} \solU\|^2_{N^{s,\gamma}_{\ell - |\beta|}}
\label{app.vw.p08}
\\
\lesssim  \noCD_{\dK}(t)+ \sqrt{\CE_{\dK,\wN}(t)}
\CD_{\dK,\wN}(t).
\end{multline}
This is the third and final estimate which we need to prove \eqref{thm.energy.1}.

We are ready to define the suitable instant energy functional as
\begin{eqnarray*}
\CE_{\dK,\wN}(t)&\eqdef& 
\sum_{|\al|\leq \dK}C_\al \|\pa^\al \solU(t)\|_{\CL^2}^2
+
\kappa_1\sum_{|\al| \le \dK-1} \|w^{\wN}\{\FI-\FP\}\pa^\al\solU\|_{\CL^2}^2
\\
&&
+
\kappa_2\sum_{|\al|= \dK} \|w^{\wN}\pa^\al\solU\|_{\CL^2}^2
+
\kappa_3\sum_{\substack{|\al|+|\be|\leq \dK\\ |\be|\geq 1}}\|w^{\wN-|\be|}\pa^\al_\be \{\FI-\FP\} \solU\|_{\CL^2}^2,
\end{eqnarray*}
for constants $0< \kappa_3\ll \kappa_2\ll \kappa_1\ll 1$ to be chosen small enough.  Notice that \eqref{def.eNm} is satisfied. The sum of  \eqref{zero.est}, \eqref{app.vw.p05}$\times \kappa_1$, \eqref{app.vw.pp05}$\times \kappa_2$ and \eqref{app.vw.p08}$\times \kappa_3$ implies that 
\begin{equation}  \notag
    \frac{d}{dt}\CE_{\dK,\wN}(t)+\la \CD_{\dK,\wN}(t)\lesssim \sqrt{\CE_{\dK,\wN}(t)}\CD_{\dK,\wN}(t).
\end{equation}
Since $\CE_{\dK,\wN}(t)$ will be sufficiently small by continuity, 
we have shown \eqref{thm.energy.1}.   \qed

\subsection{High-order energy estimates}

In this subsection, we will prove the second Lyapunov inequality in Theorem \ref{thm.energy}. Our goal is to construct a high-order instant energy functional $\CE_{\dK, \wN}^{\rm h}(t)$ satisfying \eqref{thm.energy.2} if $\CE_{\dK, \wN}(0)$ is sufficiently small (which we assume throughout this subsection). Due to \eqref{thm.energy.1}, $\CE_{\dK, \wN}(t)$ is also sufficiently small uniformly in time. Recall also the definition \eqref{def.dNm} of $\CD_{\dK, \wN}(t)$.  

\medskip

\noindent{\it Step 1.} From the system \eqref{Boltz} with \eqref{decay.ga} we obtain
\begin{equation}\label{app.ho.p01}
\dis \frac{1}{2}\frac{d}{dt}\sum_{1\leq |\al|\leq \dK}\|\pa^\al \solU\|_{\CL^2}^2+ \la \sum_{1\leq |\al|\leq \dK}
\|\pa^\al\{\FI-\FP\} \solU\|_{\spacen}^2
\lesssim \sqrt{\CE_{\dK}(t)}\CD_{\dK}(t).
\end{equation}
%
In particular we differentiate \eqref{Boltz} with $\pa^\al$, then multiply the result by $\pa^\al \solU$, integrate over 
$x,\vel$ and sum over $1\leq |\al|\leq \dK$ to obtain \eqref{app.ho.p01} using also \eqref{coerc} and \eqref{decay.ga} 
with $\wN =0$.
Multiply equation \eqref{app.vw.p04} by $\{\FI-\FP\}\solU$ to additionally get
\begin{multline}\label{app.ho.p02}
     \frac{1}{2}\frac{d}{dt}\|\{\FI-\FP\}\solU\|_{\CL^2}^2+\la \|\{\FI-\FP\}\solU\|_{\spacen}^2
     \lesssim
      \|\na_x\FP \solU\|_{\CL^2}^2
     +
     \sqrt{\CE_{\dK}(t)}\CD_{\dK}(t),
\end{multline}
and we will furthermore use \eqref{app.vw.p08} with $\wN =0$.

Recall also \eqref{coerc.m}.  Following the proof of \cite[Theorem 8.4]{gsNonCut0} we can define
$$
\tilde{\mathcal{I}}(t) 
\eqdef
\sum_{1\le |\alpha|\le \dK-1}
\left\{
\mathcal{I}_a^\alpha(t)
+
\mathcal{I}_b^\alpha(t)
+
\mathcal{I}_c^\alpha(t)\right\}.
$$
The terms 
$\mathcal{I}_a^\alpha(t)$, $\mathcal{I}_b^\alpha(t)$, and $\mathcal{I}_c^\alpha(t)$ are defined in the proof of \cite[Theorem 8.4]{gsNonCut0}.  Note that $\mathcal{I}(t)$ from \eqref{coerc.m}, as defined in \cite[(8.25)]{gsNonCut0}, is as above except that the sum is instead over
$
|\alpha|\le \dK-1.
$
It is established in the proof of \cite[Theorem 8.4]{gsNonCut0}, similar to just below \cite[(8.25)]{gsNonCut0},  that
\begin{multline}
\label{inner.t}
-\frac{d}{dt} \tilde{\mathcal{I}}(t) 
+
\sum_{1\le |\alpha|\le \dK-1}
\left(
 \| \nabla_x  \pa^\al a \|^2_{L^2_x}
 +
\|\nabla_x \pa^\al b\|^2_{L^2_x}
+ \| \nabla_x  \pa^\al c \|^2_{L^2_x}
\right)
\\
\lesssim
\sum_{|\alpha| \le \dK}\|\{{\bf I-P}\} \partial^\alpha f\|^2_{L^2_{\gamma+2s}}
+
\sqrt{\CE_{\dK}(t)}\CD_{\dK}(t).
\end{multline}
Strictly speaking, to prove \eqref{inner.t} one has to replace the use of \cite[Lemma 8.7]{gsNonCut0} in the proof of 
\cite[Theorem 8.4]{gsNonCut0} with the following estimate for $\NgE$:  
\begin{equation}
\left\| \langle \Gamma(f,f), \Rap_k\rangle\right\|_{H^{K-1}_x}^2
\lesssim
\CE_{\dK}(t)\CD_{\dK}(t).
\label{new.g.est}
\end{equation}
Here $\Rap_k$ are the exponentially decaying velocity basis vectors defined in \cite[(8.10)]{gsNonCut0}.  This is not a problem.  We use (6.12) from \cite[Proposition 6.1]{gsNonCut0} to see that 
\begin{gather*}
\left\| \langle \Gamma(f,f), \Rap_k\rangle\right\|_{H^{K-1}_x}
\lesssim
\sum_{|\alpha |\le K-1}\sum_{\alpha_1 \le \alpha}
\left\|
 \nsm \partial^{\alpha - \alpha_1} f \nsm_{L^2_{-m}} \nsm \partial^{\alpha_1} f \nsm_{L^2_{-m}} 
\right\|_{L^2_x},
\end{gather*}
which holds for any $m\ge 0$.  Then we apply \eqref{sobolev.trick} to the term in the upper bound above with fewer derivatives to obtain \eqref{new.g.est}.  It is furthermore a key point that for any functional $\CE^{\rm h}_{\dK}(t)$ satisfying \eqref{def.eNm.h}, from the proof of  \cite[Theorem 8.4]{gsNonCut0} it can be seen directly that we have the upper bound for $\tilde{\mathcal{I}}(t)$ of
\begin{equation}
\left| \tilde{\mathcal{I}}(t)  \right| \lesssim \CE^{\rm h}_{\dK}(t)\lesssim \CE^{\rm h}_{\dK,\wN}(t).
\label{key.inter.est}
\end{equation}
We will collect each of these last few estimates to prove \eqref{thm.energy.2}.

Now, we are ready to construct $\CE^{\rm h}_{\dK, \wN}(t)$. In fact, let us define
\begin{eqnarray*}
 \CE^{\rm h}_{\dK, \wN}(t) &\eqdef & \sum_{1\leq |\al|\leq \dK}\|\pa^\al \solU\|_{\CL^2}^2+\kappa_1\|\{\FI-\FP\}\solU\|_{\CL^2}^2
 \\
 &&
 +\kappa_2\sum_{\substack{|\al|+|\be|\leq \dK \\ |\be|\geq 1}}  \|w^{\wN-|\be|}\pa^\al_\be \{\FI-\FP\} \solU\|_{\CL^2}^2
 -\kappa_3 \tilde{\mathcal{I}}(t) 
   \\
 &&
 +
\kappa _4 \sum_{|\al| \le \dK-1} \|w^{\wN}\{\FI-\FP\}\pa^\al\solU\|_{\CL^2}^2
 +
\kappa _5 \sum_{|\al| = \dK} \|w^{\wN}\pa^\al\solU\|_{\CL^2}^2,
\end{eqnarray*}
for suitable constants $0<\kappa_5\ll\kappa_4\ll\kappa_3\ll \kappa_2\ll \kappa_1\ll 1$ to
be chosen sufficiently small, where $\tilde{\mathcal{I}}(t) $ satisfies \eqref{inner.t} and  \eqref{key.inter.est}.   Due to \eqref{key.inter.est}, notice that \eqref{def.eNm.h} holds true and so $\CE^{\rm h}_{\dK,\wN}(t)$ is a well-defined high-order instant energy functional. 

By choosing $0<\kappa_5\ll\kappa_4\ll\kappa_3\ll \kappa_2\ll \kappa_1\ll 1$
further small enough, the sum of \eqref{app.ho.p01},
\eqref{app.ho.p02}$\times \kappa_1$, \eqref{inner.t}$\times \kappa_3$,
\eqref{app.vw.p08}$\times \kappa_2$, 
\eqref{app.vw.p05}$\times \kappa_4$,
and
\eqref{app.vw.pp05}$\times \kappa_5$
 yields
\begin{equation}  \label{proof.ineq}
 \frac{d}{dt}\CE_{\dK,\wN}^{\rm h}(t)+\la \CD_{\dK,\wN}(t)
\lesssim
 \|\na_x\FP \solU\|_{\CL^2}^2
+
 \sqrt{\CE_{\dK,\wN}(t)}\CD_{\dK,\wN}(t).
\end{equation}
Here we have additionally added $\|\na_x\FP \solU\|^2$ to both sides of the inequality; while also recalling that $\| w^{\wN} \na_x\FP \solU\|^2 \lesssim\|\na_x\FP \solU\|^2$.  We thus establish the desired estimate \eqref{thm.energy.2} since $\CE_{\dK, \wN}(t)$ is small enough.  \qed

\subsection{The $L^2(\threed_\vel)$ nonlinear estimate}\label{sec.decayEST}  
In this subsection, we will prove the following velocity weighted $L^2(\threed_\vel)$ based norm estimates for \eqref{gamma0}.

\begin{proposition}\label{lem.h1h2}
Fix real numbers $\wB^+$, $\wB^-$, $\wB^\prime \ge 0$ with $\wB^-\ge  \wB^\prime$, and $\wB=\wB^+- \wB^-$.
We then have the the following uniform estimate 
\begin{equation}
\nsm w^{ \wB}\Ga(g,h)\nsm_{L^2} \lesssim 
\nsm w^{ \wB^+ - \wB^\prime} g\nsm_{L^2_{\gamma  + 2s}}  
\nsm w^{ \wB + \wB^\prime} h\nsm_{H^{\id}_{(\gamma+2s)+(\gamma+2s)^+}}.
\label{lem.est.g1}
\end{equation}
This will hold under \eqref{kernelP} or more generally $\gamma > -\Ndim$ combined with $\gamma+2s > -\frac{\Ndim}{2}$.

Furthermore for the soft potentials \eqref{kernelPsing} we have that  
\begin{equation}
\label{lem.est.g2}
\begin{split}
\nsm w^{\wB} \Gamma(g,h)\nsm_{L^2} \lesssim 
\nsm w^{ \wB^+ - \wB^\prime} g\nsm_{H^{\ksob}_{\gamma  + 2s}}  
\nsm w^{ \wB + \wB^\prime}  h\nsm_{H^{\id}_{(\gamma+2s)+(\gamma+2s)^+}},
\\
\nsm w^{\wB} \Gamma(g,h)\nsm_{L^2}
\lesssim 
\nsm w^{ \wB^+ - \wB^\prime} g\nsm_{L^{2}_{\gamma  + 2s}}  
\nsm w^{ \wB + \wB^\prime}  h\nsm_{H^{\id+\ksob}_{(\gamma+2s)+(\gamma+2s)^+}}.
\end{split}
\end{equation}
We again use $\ksob = \lfloor \frac{n}{2} +1 \rfloor$.  Furthermore, above and in the proof below we always have from \eqref{kernelQ} that $\id=1$ if $s\in(0,1/2)$ and $\id=2$ if $s\in [1/2,1)$.
\end{proposition}

In the proposition above, we note that $(\gamma+2s)^+ = \max\{ \gamma+2s, 0\}$.  Now, previous estimates of this type can be found in, for example, \cite{ChenHeSmoothing,MR2847536} and the references therein.  
We can not use these because they do not hold in particular under \eqref{kernelPsing}, and they do not allow negative decaying velocity weights at infinity.  
Our estimate above is not-optimal in terms of the order of differentiation, $\id$ instead of $2s$ from \eqref{kernelQ}.  It is also not optimal in terms of the order of the velocity weights, in particular because of the term $(\gamma+2s)+(\gamma+2s)^+$ above.  
However our ability to include negative decaying velocity weights allows us to obtain better dependence on the initial data in Theorem \ref{thm.ns} than would otherwise be possible (without the negative weights above, when $\wB^+ =0$ and $\wB^- > 0$, we would have to ``pay more'' with additional weights  for the decay in Theorem \ref{thm.ns}).  These estimates also have the advantage that they can be proven quickly using the machinery from \cite{gsNonCut0}.

We will use Proposition \ref{lem.h1h2} in two distinct cases. 
Firstly, for \eqref{def.eNm}, we have
\begin{equation}
 \| w^{-\ell}\Gamma(f,f) \|_{\CH^1}
 +
 \| w^{-\ell}\Gamma(f,f) \|_{ Z_1}
 \lesssim
\CE_{\dK}(t).
\label{main.u.e}
\end{equation}
This holds when $\wB^+=0$,  $\wB^- = \ell$ and $\wB^\prime = \ell /2$ where $\ell>0$ is sufficiently large.  The estimate \eqref{main.u.e} then follows from the embedding $H^{\ksob}_x \subset L^\infty_x$ and $\NgE$.
This works for either the hard or the soft potential estimates in Proposition \ref{lem.h1h2}. 

The second case is when $\wB=\wB^+\ge 0$ so that $\wB^- = \wB^\prime=0$ and we have 
\begin{equation}
 \| w^{\wB}\Gamma(f,f) \|_{\CH^{\ksob}}+
  \| w^{\wB}\Gamma(f,f) \|_{ Z_1}
 \lesssim
\CE_{\dK,\wN^\prime}(t).  
\label{main.u.e2}
\end{equation}
This now uses the $L^p-L^q$ Sobolev embeddings as in \cite[Remark of (6.9)]{gsNonCut0}.   
In \eqref{main.u.e2} for the hard potentials \eqref{kernelP} we need $\wN^\prime = b+2(\gamma+2s)$ and for the soft potentials \eqref{kernelPsing} we use $\wN^\prime = b+1$ as a result of the scaling in the weight \eqref{weigh}.

\begin{proof}
We will estimate the operator $\ang{\Gamma(g,h),f}$ and prove Proposition \ref{lem.h1h2} by duality.   We express the collision operator \eqref{gamma0} using its dual formulation from \cite[(A.1)]{gsNonCut0}.  This is explained in \cite[Proposition A.1 and (3.2)]{gsNonCut0}.  The point is that, after a transformation, we can put cancellations on the function $h$ as follows
\begin{equation}
\ang{ \Gamma(g,h),  f }
=
\int_{\threed} dv'   \int_{\threed} dv_* \int_{E_{v_*}^{v'}} d \pi_{v}  
~ \tilde{B} ~ g_*  f'    ~ \left( M_*' h - M_* h' \right)
+
\opGstar.
\label{dualOPdef}
\end{equation}
Here 
$
 v_*' = v+v_* - v'
$
and the kernel $\tilde{B}$ is given by
\begin{equation}
\tilde{B}
\eqdef
2^{n-1}
\frac{B\left(v-v_*, \frac{2v' - v- v_*}{|2v' - v- v_*|}\right) }{  |v'-v_*| ~ |v-v_*|^{n-2}}.
\label{kernelTILDE}
\end{equation}
Furthermore the operator $\opGstar=\opGstar(g,h,f)$ above  does not differentiate:
\begin{equation}
\opGstar
 \eqdef 
 \int_{\threed} dv'  f'  h'  ~    \int_{\threed} dv_*  ~  g_*   M_*  \int_{E_{v_*}^{v'}} d \pi_{v} ~  \tilde{B} 
\left(1 -  \frac{\Phi(v'-v_*) |v' - v_*|^n}{\Phi(v-v_*) |v-v_*|^n}  \right).
 \label{opGlabel}
\end{equation}
In these integrals $d\pi_{v}$ is the Lebesgue measure on the $(\Ndim - 1)$-dimensional hyperplane $E_{v_*}^{v'}$ 
defined by 
  $
  E^{v'}_{v_*} \eqdef \left\{ v\in \threed : \ang{v_*-v', v - v'} =0 \right\},
  $
and $v$ is the variable of integration.  
Furthermore let $\{ \chi_k \}_{k=-\infty}^\infty$ be a partition of unity on $(0,\infty)$ such that $\nsm \chi_k\nsm_{L^\infty} \leq 1$ and 
$\mbox{supp}\left( \chi_k \right)\subset [2^{-k-1},2^{-k}]$.  
For each $k$ we use the notation 
$$
\tilde{B}_k
\eqdef
\tilde{B}  ~ \chi_k (|v - v'|).
$$
Now we record the following decomposed pieces
\begin{equation}
\begin{split}
\teePLUSop(g,h,f)  & 
\eqdef  \int_{\threed} dv' ~   \int_{\threed} dv_* \int_{E_{v_*}^{v'}} d \pi_{v} ~ 
\tilde{B}_k ~  g_*  f' M_*' h, 
\\ 
\teeSTARop(g,h,f) & 
\eqdef
 \int_{\threed} dv' ~   \int_{\threed} dv_* \int_{E_{v_*}^{v'}} d \pi_{v} ~ 
\tilde{B}_k ~ g_* f'   M_* h'.
\end{split}
\notag
\end{equation}
We use the operators $\teePLUSop$ and $\teeSTARop$ without any weights in contrast to \cite{gsNonCut0}.  In \cite{gsNonCut0} the operators $T^{k,\ell}_{+}$, and $T^{k,\ell}_{*}$ are similar except that they include the weight $w^{2\ell}$.

Now we do a standard isotropic Littlewood-Paley decomposition 
as follows.  
Consider $\phi\in C_c^\infty(\threed)$ such that $\phi(\xi) = 0$ for $|\xi| \ge 2$ and $\phi(\xi) = 1$ for $|\xi| \le 1$.
Define $\varphi_0(\xi)=  \phi(\xi)$ and $\varphi_j(\xi)= \phi(2^{-j}\xi)- \phi(2^{-j+1}\xi)$ for $j\ge 1$ so that
$$
\sum_{j=0}^\infty \varphi_j(\xi)=1,\quad \xi\in \threed.
$$
We denote $\psi_j \eqdef \mathcal{F}^{-1}\varphi_j$ and define the notation  ``$g_j$'' 
 as follows 
$$
g_j  \eqdef (g * \psi_j)(\vel) = \int_{\threed }\psi_j(z)g(\vel-z)dz, \quad j\ge 0. 
$$
We thus have the standard expansion 
$
g  = \sum_{j=0}^\infty g_j.
$ 
Furthermore for any $\wN \in \R$ and say $m\in (0,2)$ we have the estimate 
\begin{equation}\label{l.s.equiv}
\left(\sum_{j=0}^\infty
2^{2(m-2)j}
\nsm w^\wN \left| \nabla \right|^2 g_j \nsm_{L^2(\threed_\vel)}^2 \right)^{1/2}
\lesssim
\nsm w^\wN g \nsm_{H^m(\threed_\vel)},
\end{equation}
with $\nabla$ the Euclidean gradient in $\threed$.   
Above and below we will be using the notation
$$
|\nabla|^2 g(\vel) \eqdef 
\max_{0\le j \leq 2}\sup_{|\xi| \leq 1} \left| \left(\xi \cdot \nabla \right)^j g(\vel) \right|.
$$
This will be our main tool to control the Littlewood-Paley sums below.

To proceed with our estimates, we expand the trilinear form as
\begin{align}
\ang{ \Gamma(g,h),f} & = 
 \opGstar(g,h,f)
+
\sum_{j=0}^\infty \sum_{k=-\infty}^\infty (\teePLUSop - \teeSTARop)(g,h_j,f) \nonumber \\
& = 
 \opGstar(g,h,f)
\label{sum1} 
+ \sum_{j=0}^\infty \sum_{k=-\infty}^j (\teePLUSop - \teeSTARop)(g,h_j,f) 
\\
& \hspace{30pt} + \sum_{j=0}^\infty \sum_{k=j+1}^\infty (\teePLUSop - \teeSTARop)(g,h_j,f) \label{sum2}.
\end{align}
As usual, all the sums can be rearranged because of the rapid convergence.

We first give the proof of \eqref{lem.est.g1} using several estimates from \cite{gsNonCut0}.  From \cite[Proposition 3.4]{gsNonCut0}, using 
$
\opGstar(g,h,f)
=
 \opGstar(g, w^{\weT}h,w^{-\weT}f),
$
with $\weT \in \R$ it follows that 
$$
\left| \opGstar(g,h,f) \right|
\lesssim
      \nsm g \FM^\delta \nsm_{L^2} 
  \nsm  h \FM^\delta \nsm_{H^{2s}} 
 \nsm  f \FM^\delta \nsm_{L^2} 
 + 
 \nsm g \nsm_{L^2_{-m}} 
  \nsm w^\weT h\nsm_{L^2_{\gamma}} 
 \nsm w^{-\weT} f\nsm_{L^2_{\gamma}}.
$$
Here $\delta>0$ is a small number, and $m\ge 0$ can be large.  (Note that in \cite[Proposition 3.4]{gsNonCut0}, using 
\cite[Proposition 3.5]{gsNonCut0}, it is legitimate to choose $\epsilon =s\in(0,1)$.)

To estimate the second term in \eqref{sum1} we use 
$
\teeSTARop(g,h,f)
=
\teeSTARop(g, w^{\weT} h, w^{-\weT} f)
$
again. 
Then from 
 \cite[Proposition 3.2]{gsNonCut0}
 we have the estimate
 $$
  \left| \teeSTARop(g,h,f) \right| 
  \lesssim 
  2^{2sk}  \nsm g\nsm_{L^2_{-m}} \nsm w^{\weT} h\nsm_{L^2_{\gamma+2s}}
 \nsm w^{-\weT} f\nsm_{L^2_{\gamma+2s}}.
 $$
 Unfortunately the estimate for $\teePLUSop$ can not use the same trick as easily.  However following 
the proof of \cite[Proposition 3.3]{gsNonCut0} we see that for any $\weT^+, \weT^-, \weT' \ge 0$ with $\weT = \weT^+ - \weT^-$ and $\weT^-\ge \weT'$ we have that
$$
\left|  \teePLUSop(g,h,f)  \right| \lesssim   
2^{2sk} 
 \nsm  w^{\weT^+ - \weT'} g\nsm_{L^2} 
\nsm w^{\weT + \weT'} h\nsm_{L^2_{\gamma + 2s}} 
\nsm w^{-\weT} f\nsm_{L^2_{\gamma + 2s}}.
$$
This can be derived quickly from the proof of \cite[Proposition 3.3]{gsNonCut0} as follows.  When $\gamma + 2s < 0$ as in \eqref{kernelP} in \cite[(3.14)]{gsNonCut0} we replace $w^{2\ell}(v')$ with $w^{\weT}(v') w^{-\weT}(v')$.  Then in the top factor of \cite[(3.15)]{gsNonCut0}  $w^{2\ell}(v')$ is replaced by  $w^{2\weT}(v')$ and in the bottom factor of \cite[(3.15)]{gsNonCut0} $w^{2\ell}(v')$ is replaced by  $w^{-2\weT}(v')$ and the rest of the proof is exactly the same under \eqref{kernelP}.  When alternatively $\gamma + 2s \ge 0$, in \cite[(3.16)]{gsNonCut0}
we can replace $w^{4\ell}(v')$ with $w^{2\weT}(v') w^{-2\weT}(v')$ and we replace $w^{2(\ell+\ell')}(v')$ in both places with 
$w^{2(\weT+\weT')}(v')$.  We put $w^{-2\weT}(v')$ with $|f'|^2$ in \cite[(3.16)]{gsNonCut0} and the rest of the proof does not change. This $\teePLUSop$ estimate is the largest of the two above.

Then for the second term in \eqref{sum1} we obtain the upper bound of
\begin{gather*}
\sum_{j=0}^\infty \sum_{k=-\infty}^j  \left| (\teePLUSop - \teeSTARop)(g,h_j,f) \right|
\lesssim  
\sum_{j=0}^\infty
2^{2sj} 
 \nsm  w^{\weT^+ - \weT'} g\nsm_{L^2} 
\nsm w^{\weT + \weT'} h_j\nsm_{L^2_{\gamma + 2s}} 
\nsm w^{-\weT} f\nsm_{L^2_{\gamma + 2s}}
\\
\lesssim  
 \nsm  w^{\weT^+ - \weT'} g\nsm_{L^2} 
\nsm w^{\weT + \weT'} h\nsm_{H^{\id}_{\gamma + 2s}} 
\nsm w^{-\weT} f\nsm_{L^2_{\gamma + 2s}}.
\end{gather*}
The last inequality used \eqref{l.s.equiv} as well as 
$2sj = \id j +(\id-2s)j$ with \eqref{kernelQ} and $(\id-2s) <0$.

To estimate \eqref{sum2} we have to exploit the cancellations.   Following the proof of \cite[Proposition 3.7]{gsNonCut0} for $\weT \in \R$ we obtain the estimate (for $m \ge 0$ large)
$$
\left| (\teePLUSop - \teeSTARop)(g,h_j,f) \right|
\lesssim 2^{(2s-2)k} \nsm g\nsm_{L^2_{-m}} 
\nsm  w^{\weT} \left| \nabla \right|^2 h_j \nsm_{L^2_{\gamma+2s}}
\nsm w^{-\weT} f\nsm_{L^2_{\gamma+2s}}.  
$$
To prove this estimate with the isotropic Euclidean derivative $\na$ as above follows directly the proof of \cite[Proposition 3.7]{gsNonCut0} except that we use 
$$
h_j(v')-h_j(v)=(v-v')\cdot(\na h_j)(v')+\int_0^1 d \Ctheta ~
(v'-v)\otimes(v'-v):(\na^2 h_j)(\CurveP(\Ctheta)), 
$$
where $\CurveP(\Ctheta)=\Ctheta v'+(1-\Ctheta)v$.  Note that 
$
\int_{E_{v_*}^{v'}} d \pi_{v} \tilde{B}_k ~ (v'-v)\cdot(\na h_j)(v')
=0
$
by symmetry, a longer explanation is given in \cite[(3.41)]{gsNonCut0}.  We further have
\begin{equation}
\left| h_j(v')-h_j(v)-(v-v')\cdot(\na h_j)(v')\right| 
 \lesssim |v-v'|^2 \int_0^1 d \Ctheta  ~ (| \nabla|^2 h_j) (\CurveP(\Ctheta)). 
\label{paradiff2}
\end{equation}
Therefore in the proof of \cite[Proposition 3.7]{gsNonCut0}
we replace the use of \cite[(3.25)]{gsNonCut0} by instead using \eqref{paradiff2}.
The only other difference is that we replace the weight $w^{2\ell}(v')$ with $w^{\weT}(v')w^{-\weT}(v')$.  The factor $w^{\weT}(v')$ follows the function $h_j$ and the factor $w^{-\weT}(v')$ goes with the function $f$.  Otherwise the proof is exactly the same as \cite[Proposition 3.7]{gsNonCut0}.   Now we estimate \eqref{sum2} with the upper bound of
\begin{gather*}
\sum_{j=0}^\infty \sum_{k=j+1}^\infty  \left| (\teePLUSop - \teeSTARop)(g,h_j,f) \right|
\lesssim  
\sum_{j=0}^\infty
 2^{(2s-2)j} \nsm g\nsm_{L^2_{-m}} 
\nsm  w^{\weT} \left| \nabla \right|^2 h_j \nsm_{L^2_{\gamma+2s}}
\nsm w^{-\weT} f\nsm_{L^2_{\gamma+2s}}
\\
\lesssim  
 \nsm   g\nsm_{L^2_{-m}} 
\nsm w^{\weT} h\nsm_{H^{\id}_{\gamma + 2s}} 
\nsm w^{-\weT} f\nsm_{L^2_{\gamma + 2s}}
\\
\lesssim  
 \nsm  w^{\weT^+ - \weT'} g\nsm_{L^2} 
\nsm w^{\weT + \weT'} h\nsm_{H^{\id}_{\gamma + 2s}} 
\nsm w^{-\weT} f\nsm_{L^2_{\gamma + 2s}}.
\end{gather*}
Of course, these inequalities used \eqref{l.s.equiv}.

To prove \eqref{lem.est.g1} it remains to choose the weights.   
First suppose that \eqref{kernelP} holds so that $\gamma + 2s \ge 0$.  In this case from \eqref{weigh} we have that $w(v)= \ang{v}$.  We choose
$$
\weT^+ = (\gamma+2s) +\wB^+,
\quad
\weT^- = \wB^-,
\quad\
\weT' = \wB'.
$$
Alternatively, when $\gamma + 2s < 0$ but $\gamma + 2s > -\frac{n}{2}$ with \eqref{kernelPsing}, since in this case under \eqref{weigh} we have
$
w=\ang{\vel}^{-\gamma-2s},
$
we can choose 
\begin{equation}
\weT^+ =  \wB^+,
\quad
\weT^- = 1+\wB^-,
\quad\
\weT' =1+ \wB^\prime.
\label{exp.neg}
\end{equation}
Then, in either case, collecting these upper bounds we have shown that
\begin{equation}
\left| \ang{\Gamma(g,h),f} \right|
\lesssim
\nsm w^{ \wB^+ - \wB^\prime}  g\nsm_{L^2_{\gamma  + 2s}}  
\nsm w^{ \wB + \wB^\prime}  h\nsm_{H^{\id}_{(\gamma+2s)+(\gamma+2s)^+ }}
\nsm w^{ -\wB}  f\nsm_{L^{2}}.
\notag
\end{equation}
The estimate in \eqref{lem.est.g1} then follows by duality.

We turn to the estimate for \eqref{lem.est.g2} which will hold under \eqref{kernelPsing}.  We sketch the arguments and the estimates that we use below.  The slight modifications are all similar to the case of \eqref{lem.est.g1}.  From \cite[Proposition 4.4]{gsNonCut0} we have for any $\weT\in\R$ that 
\begin{gather*}
\left| \opGstar (g,h,f)   \right|  \lesssim
      \nsm g\nsm_{H^{\ksob}_{-m}} 
  \nsm w^{\weT} h\nsm_{L^2_{\gamma}} 
 \nsm w^{-\weT} f\nsm_{L^2_{\gamma}},
 \label{tstarCf} 
 \\
  \left| \opGstar(g,h,f) \right|   \lesssim  
 \nsm g\nsm_{L^2_{-m}} 
| w^{\weT} h |_{H^{\ksob}_{\gamma}} | w^{-\weT} f |_{L^2_{\gamma}}.
 \label{tstarCh}
\end{gather*}
Similarly from \cite[Proposition 4.2]{gsNonCut0} we have for any $\weT\in\R$ that 
\begin{gather*}
  \left| \teeSTARop(g,h,f) \right|   \lesssim 
  2^{2sk}    \nsm g\nsm_{H^{\ksob}_{-m}} 
  \nsm w^{\weT} h\nsm_{L^2_{\gamma + 2s}} 
 \nsm w^{-\weT} f\nsm_{L^2_{\gamma + 2s}},
 \label{tstarg} 
 \\
  \left| \teeSTARop(g,h,f) \right|   \lesssim 2^{2sk} 
 \nsm g\nsm_{L^2_{-m}} 
| w^{\weT} h |_{H^{\ksob}_{\gamma + 2s}} 
| w^{-\weT} f |_{L^2_{\gamma + 2s}}.
 \label{tstarh}
\end{gather*}
Then again fix  any $\weT^+, \weT^-, \weT' \ge 0$,
with $\weT = \weT^+ - \weT^-$ and   $ \weT' \le \weT^-$.  Then from \cite[Proposition 4.3]{gsNonCut0} we have the uniform estimates 
\begin{gather*}
\left|  \teePLUSop(g,h,f)  \right| \lesssim    
2^{2sk} 
 \nsm w^{\weT^+ - \weT'} g\nsm_{H^{\ksob}} 
\nsm w^{\weT + \weT'} h\nsm_{L^2_{\gamma + 2s}} 
\nsm w^{-\weT} f\nsm_{L^2_{\gamma + 2s}},
 \label{tplussmallN}
\\
\left|  \teePLUSop(g,h,f)  \right| \lesssim    
2^{2sk} 
\nsm w^{\weT^+ - \weT'} g\nsm_{L^2} 
 \nsm w^{\weT + \weT'} h\nsm_{H^{\ksob}_{\gamma + 2s} }
\nsm w^{-\weT} f\nsm_{L^2_{\gamma + 2s}}.
 \label{tplussmall2N}
\end{gather*}
On the other hand from \cite[Proposition 4.6]{gsNonCut0} we have that
\begin{gather*}
\label{cancelh2g1}
 \left| (\teePLUSop- \teeSTARop)(g,h,f) \right| 
   \lesssim 2^{(2s-2)k} 
    \nsm g\nsm_{L^2_{-m}} 
\nsm w^\weT  |\tilde{\nabla}|^2 h \nsm_{H^{\ksob}_{\gamma+2s}} 
  \nsm w^{-\weT} f\nsm_{L^2_{\gamma+2s}},
\\
 \left| (\teePLUSop- \teeSTARop)(g,h,f) \right| 
   \lesssim 2^{(2s-2)k} 
\nsm  g\nsm_{H^{\ksob}_{-m}}
  \nsm w^\weT |\tilde{\nabla}|^2 h \nsm_{L^2_{\gamma+2s}}
  \nsm w^{-\weT} f\nsm_{L^2_{\gamma+2s}}.
  \label{cancelh2g2}
\end{gather*}
The (slight) modifications in the proofs from \cite{gsNonCut0} required to obtain each of these estimates is exactly the same as was explained in the analogous estimate from the prior proof of \eqref{lem.est.g1}.   Following the same summation procedure as in \eqref{lem.est.g1} yields
\begin{gather*}
\left| \ang{\Gamma(g,h),f} \right|
\lesssim
\nsm w^{ \weT^+ - \weT^\prime}  g\nsm_{H^{\ksob}}  
\nsm w^{ \weT + \weT^\prime}  h\nsm_{H^{\id}_{\gamma+2s}}
\nsm w^{ -\weT}  f\nsm_{L^2_{\gamma+2s}},
\\
\left| \ang{\Gamma(g,h),f} \right|
\lesssim
\nsm w^{ \weT^+ - \weT^\prime}  g\nsm_{L^2}  
\nsm w^{ \weT + \weT^\prime}  h\nsm_{H^{\id+\ksob}_{\gamma+2s}}
\nsm w^{ -\weT}  f\nsm_{L^2_{\gamma+2s}}.
\end{gather*}
To finish \eqref{lem.est.g2}, we choose the weights as in \eqref{exp.neg} and again use duality.
\end{proof}

We have now completed all of the Lyapunov energy estimates and $L^2$ norm estimates which are needed.  In the remaining subsections we will use these to prove the time decay estimates from Theorem \ref{thm.ns} and Corollary \ref{cor}.

\subsection{Decay rates of the full instant energy functional}\label{sec.decayNL.1}
This subsection is devoted to the proof of \eqref{thm.ns.2}.  
Recall \eqref{weigh} and \eqref{split.E}, where again $p^\prime\ge 0$ will be chosen later.   From \eqref{def.eNm} let 
$\CE_{\dK,\wN}^{low}(t)$ be the restriction of $\CE_{\dK,\wN}(t)$ to $E(t)$ and similarly define 
$\CE_{\dK,\wN}^{high}(t)$ as the restriction to  $E^c(t)$ in \eqref{split.E}.  We begin by studying only the case of the soft potentials \eqref{kernelPsing}.  Afterwards, we explain how to obtain the hard potential \eqref{kernelP} decay estimates.

From \eqref{def.dNm} and \eqref{def.eNm} for the 
soft potentials \eqref{kernelPsing}, with $p^\prime>0$, we have
\begin{equation}
\frac{\CE_{\dK,\wN}^{low}(t)}{ t^{p^\prime} }\lesssim \frac{\|\FP \solU\|_{\CL^2}^2}{ t^{p^\prime} }+  \CD_{\dK,\wN}(t).
\label{soft.up.d}
\end{equation}
Adding this into \eqref{thm.energy.1}, we therefore conclude 
\begin{equation*}
\frac{d}{dt}\CE_{\dK,\wN}(t)
+
\lambda p t^{p-1} \CE_{\dK,\wN}(t)
\lesssim
t^{p-1}
\left(
\|\FP \solU\|_{\CL^2}^2
+
 \CE_{\dK,\wN}^{high}(t)
 \right),
\end{equation*}
where we defined $\lambda= \frac{\la^\prime}{ p } $ with $p= -p^\prime+1>0$.  Use the factor $e^{-\lambda t^{p}}$ to obtain
\begin{equation}
\CE_{\dK,\wN}(t) 
\lesssim
e^{-\lambda t^{p}}\CE_{\dK,\wN}(0) +   
\int_0^t 
ds ~ e^{-\lambda (t^{p}-s^p)}
s^{p-1}
\left( 
\|\FP \solU\|_{\CL^2}^2
+
\Ac  \CE_{\dK,\wN}^{high}(s)
\right).
\label{main.gs.split}
\end{equation}
We suppose $\Ac\ge 0$ and $p>0$, equivalently $0< p^\prime<1$,  so that the integral is finite. 

In this case, 
since $\CE_{\dK,\wN}^{high}(s)$ is restricted to $E^c(s)$  we have that
\begin{equation}
\label{decay.high}
\CE_{\dK,\wN}^{high}(t)
\lesssim
(1+t)^{-\frac{\Ndim}{2}} \CE_{\dK,\wN+\AsoftI}(0),
\end{equation}
where we have used 
$
1  \lesssim w^{\frac{\Ndim}{2p^\prime}}(v) (1+t)^{-\frac{\Ndim}{2}}
$
on $E^c(s)$ and \eqref{thm.energy.1}.  For Theorem \ref{thm.ns}, note that we can choose
$
p^\prime
$
very close to one so that 
$
\AsoftI \eqdef \frac{\Ndim}{2p^\prime} > \frac{\Ndim}{2}.
$

For now, we leave that be and study the time decay under either \eqref{kernelP} or \eqref{kernelPsing}.
We consider the pointwise time decay estimates on $\|\FP \solU\|_{\CL^2}^2$. Formally, the solution $\solU$ to the Cauchy problem \eqref{Boltz} of the Boltzmann equation can be written as \eqref{ls.semi}
where
$\sourceG$ is given by \eqref{def.g.non}.
We conclude that
\begin{equation}  \label{Boltz.rep}
   \solU(t)=I_0(t)+I_1(t),
\end{equation}
with 
$$
  I_0(t) =  \semiG(t)\solU_0, \quad
  I_1(t) =\int_0^t\semiG(t-s)\Gamma(\solU,\solU)(s)ds.
$$
We recall the norms from \eqref{brief.norm}.  We now apply \eqref{thm.ls.1} with $m=0$, $r=1$ and $\wN=-\wB\le 0$ to be determined 
 to $I_0(t)$ and $I_1(t)$, respectively, to obtain
\begin{equation}  \notag
    \|w^{-\wB}I_0(t)\|_{\CL^2}
    \lesssim
    (1+t)^{-\frac{\Ndim}{4}} \|w^{-\wB+\wE} \solU_0\|_{{\CL^2}\cap Z_1},
\end{equation}
and
\begin{multline*}
    \|w^{-\wB} I_1(t)\|_{\CL^2}\leq \int_0^t\|w^{-\wB}\semiG(t-s)\Gamma(\solU,\solU)(s)\|_{\CL^2}ds
    \\
\lesssim
    \int_0^t(1+t-s)^{-\frac{\Ndim}{4}}\|w^{-\wB+\wE}\Gamma(\solU,\solU)(s)\|_{{\CL^2}\cap Z_1}ds.
\end{multline*}
Under  \eqref{kernelP} we can take $\wE=0$ and for \eqref{kernelPsing} we take any $\wE > \frac{\Ndim}{2}$.  Define
\begin{equation}\label{def.einfty}
    \CE_{\dK,\wN}^\infty(t) \eqdef \sup_{0\leq s\leq t}(1+s)^{\frac{\Ndim}{2}}\CE_{\dK,\wN}(s).
\end{equation}
For $I_1(t)$, from \eqref{main.u.e} and the definition \eqref{def.einfty} of $\CE_{\dK,\wN}^\infty(t)$, it holds that
\begin{multline*}
    \|w^{-\wB} I_1(t)\|_{\CL^2}\lesssim
    \int_0^t(1+t-s)^{-\frac{\Ndim}{4}}\CE_{\dK}(s)ds
    \lesssim
    \int_0^t(1+t-s)^{-\frac{\Ndim}{4}}\CE_{\dK,\wN}(s)ds\\
    \lesssim \CE_{\dK,\wN}^\infty(t)\int_0^t(1+t-s)^{-\frac{\Ndim}{4}}(1+s)^{-\frac{\Ndim}{2}}ds
    \lesssim \CE_{\dK,\wN}^\infty(t)(1+t)^{-\frac{\Ndim}{4}}.
\end{multline*}
Here we have chosen $\wB>0$ sufficiently large and used $\NgE$ in \eqref{main.u.e}.  We have also used the decay estimates for the time integrals as in \cite[Proposition 4.5]{strainSOFT}.

Collecting the estimates on $I_1(t)$ and $I_2(t)$ above, with \eqref{form.p}, implies 
\begin{multline}
\|\FP \solU(t)\|_{\CL^2}^2\lesssim\|w^{-\wB}\solU(t)\|_{\CL^2}^2\lesssim
\|w^{-\wB} I_0(t)\|_{\CL^2}^2
+
\|w^{-\wB} I_1(t)\|_{\CL^2}^2
\\
\lesssim
(1+t)^{-\frac{\Ndim}{2}} \|\solU_0\|_{\CL^2\cap Z_1}^2
+
 (1+t)^{-\frac{\Ndim}{2}}[\CE_{\dK,\wN}^\infty(t)]^2.\label{thm.ns.p2}
\end{multline}
Note that this holds under either either \eqref{kernelP} or \eqref{kernelPsing}.  

Now we plug \eqref{thm.ns.p2} and
\eqref{decay.high} 
into 
\eqref{main.gs.split} to obtain \eqref{thm.ns.2} in Theorem \ref{thm.ns} for the soft potentials \eqref{kernelPsing}
by a bootstrap argument since $\eps_{\dK,\wN+\AsoftI}$ is sufficiently small.  

For the hard potentials \eqref{kernelP}, 
we can replace \eqref{soft.up.d} with 
$$
\CE_{\dK,\wN}(t) \lesssim \|\FP \solU\|_{\CL^2}^2+  \CD_{\dK,\wN}(t).
$$
Then we obtain \eqref{main.gs.split} 
with $p=1$, except that in the upper bound $\Ac =0$.
As in the previous case, we plug in \eqref{thm.ns.p2}
and use a bootstrap argument with $\eps_{\dK,\wN}$ sufficiently small
to obtain \eqref{thm.ns.2} in Theorem \ref{thm.ns} for the hard potentials  \eqref{kernelP}.
 \qed

\subsection{Decay rates of the high-order instant energy functional}\label{sec.decayNL.2} We will now prove the faster time-decay estimate \eqref{thm.ns.3} for the high-order instant energy functional $\CE_{\dK,\wN}^{\rm h}(t)$ from \eqref{def.eNm.h}. We will use \eqref{thm.energy.2} from Theorem \ref{thm.energy}.  As in the last subsection, initially we restrict to the case of the soft potentials \eqref{kernelPsing}.

With \eqref{split.E} and \eqref{def.eNm.h} we define $\CE_{\dK,\wN}^{{\rm h},low}(t)$ to be the restriction of 
$\CE_{\dK,\wN}^{\rm h}(t)$ to $E(t)$ and similarly $\CE_{\dK,\wN}^{{\rm h},high}(t)$ is the restriction to $E^c(t)$.  Then with \eqref{def.dNm}:
\begin{equation}
\label{soft.up.h}
\frac{\CE_{\dK,\wN}^{{\rm h},low}(t)}{ t^{p^\prime} }\lesssim   \CD_{\dK,\wN}(t).
\end{equation}
Now we plug this into \eqref{thm.energy.2}  to conclude that 
\begin{equation}
\notag
    \frac{d}{dt} \CE_{\dK,\wN}^{\rm h}(t)+ \frac{\la^\prime}{ t^{p^\prime} }  \CE_{\dK,\wN}^{{\rm h}}(t)
    \lesssim
    \|\na_x\FP \solU(t)\|_{\CL^2}^2
    +
     \frac{\CE_{\dK,\wN}^{{\rm h},high}(t)}{ t^{p^\prime} }.
\end{equation}
Following the exact procedure used to obtain \eqref{main.gs.split} we achieve
\begin{multline}
\CE_{\dK,\wN}^{\rm h}(t)
\lesssim
e^{-\lambda t^{p}}\CE_{\dK,\wN}^{\rm h}(0) 
\\
+   
\int_0^t 
ds ~ e^{-\lambda (t^{p}-s^p)}
\left( 
\|\na_x\FP \solU(s)\|_{\CL^2}^2
+
\Ac
s^{p-1} 
 \CE_{\dK,\wN}^{{\rm h},high}(s)
\right).
\label{main.h.split}
\end{multline}
Recall $p = -p^\prime +1>0$ and $\Ac \ge 0$.
Again, on $E^c(s)$  we have
\begin{equation}
\label{decay.high.h}
\CE_{\dK,\wN}^{{\rm h},high}(t)
\lesssim
(1+t)^{-\frac{\Ndim+2}{2}} \CE_{\dK,\wN+\Asoftq}(0).
\end{equation}
We similarly used 
$
1  \lesssim w^{\frac{\Ndim+2}{2p^\prime}}(v) (1+t)^{-\frac{\Ndim+2}{2}}
$
on $E^c(s)$, \eqref{thm.energy.1}, and
$
\Asoftq \eqdef \frac{\Ndim+2}{2p^\prime} > \frac{\Ndim+2}{2}.
$

We will pause that line of reasoning for a moment, and consider time decay under both \eqref{kernelP} and \eqref{kernelPsing}.
It follows from Theorem \ref{thm.ls} and \eqref{form.p} for $\wB \ge 0$ that
\begin{multline}
\notag
\|\na_x\FP \solU(t)\|_{\CL^2}^2
\lesssim 
\|w^{-\wB}\na_x\FP \solU(t)\|_{\CL^2}^2
\lesssim 
(1+t)^{-\frac{\Ndim+2}{2}} \|w^{-\wB+\wE}\solU_0\|_{\CHd^{1}\cap Z_1}^2\\
 + \int_0^t (1+t-s)^{-\frac{\Ndim+2}{2}} \|w^{-\wB+\wE} \Gamma(\solU,\solU)(s)\|^2_{\CHd^{1}\cap Z_1}ds.
\end{multline}
Here $\wE$ is defined as in Theorem \ref{thm.ls}.  
We then use \eqref{main.u.e} to see that for $b>0$ sufficiently large one has  
\begin{equation}
 \|w^{-\wB+\wE} \Gamma(\solU,\solU)(s)\|^2_{\CHd^{1}\cap Z_1}
 \lesssim
[\CE_{ \dK}(s)]^2.
\notag
\end{equation}
Now using the notation in Theorem \ref{thm.ns}, we choose $\AsoftI \le \Asoftq$ so that we have $\eps_{\dK,\AsoftI}\le \eps_{\dK,\Asoftq}$ is sufficiently small as in \eqref{def.id.jm}.
We can then plug the upper bound for $\CE_{ \dK}(t)$ in  \eqref{thm.ns.2} into this chain of inequalities.  These estimates yield
\begin{equation}  \label{dr.he.p03}
\|\na_x\FP \solU(t)\|_{\CL^2}^2
\lesssim
\eps_{\dK,\Asoftq} (1+t)^{-\frac{\Ndim+2}{2}}.
\end{equation}
This holds either under \eqref{kernelP} or \eqref{kernelPsing}.  

For the soft potentials \eqref{kernelPsing}, we can plug \eqref{dr.he.p03} and \eqref{decay.high.h} into \eqref{main.h.split} to obtain \eqref{thm.ns.3} in Corollary \ref{cor} where $\varepsilon = \varepsilon(p^\prime)>0$ can be arbitrarily small.  It remains to consider the hard potentials \eqref{kernelP}.  Under \eqref{kernelP}, again, we can replace \eqref{soft.up.h} with 
$$
\CE_{\dK}^{{\rm h}}(t) \lesssim   \CD_{\dK}(t).
$$
Then we can obtain \eqref{main.h.split} with $p=1$ and $\Ac =0$.
As in the previous analysis, we plug in \eqref{dr.he.p03}  with $\eps_{\dK,0}$ sufficiently small and $\Asoftq =0$
to obtain \eqref{thm.ns.3}.
\qed

\subsection{Time decay rates of solutions in $L^r_x$}\label{sec.decayNL.3} In this final subsection, we will prove both 
\eqref{cor.1} from Theorem \ref{thm.ns} and 
\eqref{cor.2} from Corollary \ref{cor}. Notice already that the time decay rates stated in \eqref{cor.1} and \eqref{cor.2} are both true when $r=2$ due to the  assumptions of Corollary \ref{cor}, Theorem \ref{thm.ns} and the definitions \eqref{def.eNm} and \eqref{def.eNm.h} of $\CE_{\dK,\wN}(t)$ and $\CE_{\dK,\wN}^{\rm h}(t)$ respectively.  

It then remains to verify the following {\it claim}:
\begin{equation}\label{dr.inf.0}
  \|\solU(t)\|_{Z_\infty}= \|\solU(t)\|_{L^2_\vel(L^\infty_x)}\lesssim (1+t)^{-\frac{\Ndim}{2}}.
\end{equation}
Once this is done \eqref{cor.1} and \eqref{cor.2} follow from interpolation. We begin with

\begin{lemma}\label{lem.semi.inf}
Using \eqref{ls.semi}, we have the following estimates uniformly in $t\geq 0$:
\begin{equation}\notag
    \|\semiG (t)\solU_0\|_{Z_\infty}
    \lesssim
    (1+t)^{-\frac{\Ndim}{2}}\|w^{\jN}\solU_0\|_{{\CH^{\ksob}}\cap Z_1}.
\end{equation}
Above the weight power is given by any $\jN>\frac{\Ndim}{2}+\ksob$ for the soft potentials  \eqref{kernelPsing} and 
$\jN=0$ for the hard potentials  \eqref{kernelP}.  
Furthermore, it holds that
\begin{equation}\notag
    \|\{\FI - \FP\}\semiG (t)\solU_0\|_{Z_\infty}
    \lesssim
    (1+t)^{-\frac{\Ndim+1}{2}+\epsilon}\|w^{\jN'}\solU_0\|_{{\CH^{\ksob+1}}\cap Z_1}.
\end{equation}
In this case the weight power is given by
$\jN'=0$ (and $\epsilon=0$) for the hard potentials  \eqref{kernelP}.  
 And for the soft potentials  \eqref{kernelPsing} with any small $\epsilon >0$ we have
 $\jN' = \jN'(\epsilon)>0$ is sufficiently large.  We therefore observe that the microscopic part can have faster linear $L^\infty_x$ decay.  
\end{lemma}

Now we see that the nonlinear decay rate in \eqref{cor.2} for $\{\FI-\FP\}\solU(t)$ is not optimal in $Z_\infty$, at least in comparison to the linear decay in $Z_\infty$ from Lemma \ref{lem.semi.inf} above.

\begin{proof}
Set $\solU^I(t)=\semiG (t)\solU_0$. From the Sobolev inequality \cite[Proposition
3.8]{MR1477408}:  
\begin{equation}\label{lem.semi.inf.p1}
\begin{split}
  \|\solU^I\|_{L^2_\vel(L^\infty_x)} 
  &\lesssim 
   \|\na_x^{k+1} \solU^I\|_{\CL^2}^{1/2}\|\na_x^{k} \solU^I\|_{\CL^2}^{1/2}, \quad
   \Ndim = 2k+1,
   \\
     \|\solU^I\|_{L^2_\vel(L^\infty_x)}
     &\lesssim 
   \|\na_x^{k+1} \solU^I\|_{\CL^2}^{1/2}\|\na_x^{k-1} \solU^I\|_{\CL^2}^{1/2},
   \quad
   \Ndim = 2k.
   \end{split}
\end{equation}
Note that in either case $k+1 = \ksob = \lfloor \frac{n}{2} +1 \rfloor$.
We now apply Theorem \ref{thm.ls} with $m$, $r=1$, and $\wN=0$ to observe the following decay rates
\begin{equation*}
    \|\na_x^{m} \solU^I(t)\|_{\CL^2}
    \lesssim
     (1+t)^{-\frac{\Ndim+2m}{4}} \|w^{\wE}\solU_0\|_{{\CHd^{m}}\cap Z_1}.
\end{equation*}
Under \eqref{kernelP} we use $\wE=0$ and with \eqref{kernelPsing} we choose any $\wE>\frac{\Ndim+2m}{2}$.  
Collecting the appropriate estimates for the correct $m$ in \eqref{lem.semi.inf.p1} grants the first estimate in Lemma \ref{lem.semi.inf}.  For the second estimate, use $\solU^I=\{\FI - \FP\}\semiG (t)\solU_0$ and apply Corollary \ref{cor.ls}.  
\end{proof}

Now, by applying Lemma \ref{lem.semi.inf} to the representation \eqref{Boltz.rep} of $\solU(t)$, one has
\begin{eqnarray}
 \nonumber 
  \|\solU(t)\|_{Z_\infty} &\lesssim &  (1+t)^{-\frac{\Ndim}{2}}\|w^{\jN}\solU_0\|_{\CH^{\ksob}\cap Z_1}\\
  &&+\int_0^t(1+t-s)^{-\frac{\Ndim}{2}}\|w^{\jN} \Ga(f,f)(s)\|_{\CH^{\ksob}\cap Z_1} ds.
  \notag
\end{eqnarray}
Then \eqref{main.u.e2} tells us that 
\begin{equation*}
\|w^{\jN} \Ga(f,f)(t)\|_{\CH^{\ksob}\cap Z_1} \lesssim  \CE_{\dK,\lN}(t).
\end{equation*}
Here $\lN$ is defined in \eqref{exponent.ln}.  
Now since we suppose that $\eps_{\dK,\lN+\AsoftI}$ in \eqref{def.id.jm} is sufficiently small, it follows from \eqref{thm.ns.2} that  
\begin{equation*}
    \CE_{\dK,\lN}(t)
    \lesssim
    \eps_{\dK,\lN+\AsoftI}(1+t)^{-\frac{\Ndim}{2}}.
\end{equation*}
Collecting the estimates in this paragraph grants  \eqref{dr.inf.0}.   \hfill {\bf Q.E.D.}

\appendix
\section{Time decay from H{\"{o}}lder and Hausdorff-Young}
\label{secAPP:HHY}

It is now our goal to prove the following inequality (for $t \ge 1$, and $1\leq r \leq 2$):
\begin{equation}
\int_{\threed}d\vel~
 \int_{|k|\leq 1} dk ~ 
 |k|^{2m} h(|k|^2t) ~ |\hat{g}(k,\vel)|^2
 \leq  C(h) (1+t)^{-2\sigma_{r,m}}
\|g\|_{Z_r}^2.
\label{ineqW}
\end{equation}
Here $C(h)>0$ and $\sigma_{r,m}$  is defined in \eqref{rateLIN}.
We make the assumption that
\begin{equation}
\left| h(s) \right| 
\le
C(h)
(1+s)^{-j},
\quad
\forall s \ge 0,
\quad
\exists
j > 2\sigma_{r,m}.
\label{hASSUME}
\end{equation}
Inequalities such as this one can be seen in \cite{MR1057534} and e.g. \cite{MR1379589,arXiv:0912.1742}.  However we further establish \eqref{ineqW} under the more general assumption \eqref{hASSUME}.

To prove \eqref{ineqW} when $r=1$ notice that 
$$
\int_{\threed}d\vel~
\sup_{|k|\leq 1} |\hat{g}(k,\vel)|^2
 \leq  C
\|g\|_{Z_1}^2.
$$
On the other hand 
$$
 \int_{|k|\leq 1} dk ~ 
 |k|^{2m}
 h(|k|^2t)
\le  C(h) t^{-n/2-m}
= C\|h \|_{L^1_{2m}(\threed)}  t^{-\sigma_{1,m}}.
$$
This establishes \eqref{ineqW} for $r=1$.

Now to prove \eqref{ineqW} when $1<r< 2$ we use the following two inequalities.  The H\"{o}lder inequality is of course with $\frac{1}{p}+\frac{1}{q} = 1$ that
$$
\int_{\threed} |g(x) h(x)| ~ dx
\le 
\left(\int_{\threed} |g(x)|^q dx \right)^{1/q}
\left(\int_{\threed} |h(x)|^p dx \right)^{1/p}.
$$
The Hausdorff-Young inequality can be expressed as
$$
\left(\int_{\threed} |\hat{g}(k)|^q dk \right)^{1/q}
\le
C(p)
\left(\int_{\threed} |g(x)|^p dx \right)^{1/p},
$$
which will hold for any 
$1\le p \le 2$
and
 $\frac{1}{p}+\frac{1}{q} = 1$
 or $q=p/(p-1)$.  Now to get back to \eqref{ineqW}, when $1<r< 2$,  we apply the H\"{o}lder inequality as
$$
 \int_{|k|\leq 1} dk ~ 
 |k|^{2m}
 h(|k|^2t) ~ |\hat{g}(k,\vel)|^2
\lesssim
\left( \int_{|k|\leq 1} dk ~ |\hat{g}(k,\vel)|^{r/(r-1)} \right)^{(r-1)/r}
A.
$$
Above we are using the definition
\begin{multline*}
A \eqdef  
\left( 
\int_{|k|\leq 1} dk ~ 
 |k|^{2m\left( \frac{r}{2-r}\right) } \left|  h(|k|^2 t) \right|^{\left( \frac{r}{2-r}\right) }
 \right)^{\frac{2}{r}-1}
 \\
=
 t^{-\left(\frac{n}{2}+ m\left( \frac{r}{2-r}\right) \right)\left( \frac{2}{r}-1 \right)}
 \left( 
\int_{|k|\leq \sqrt{t}} dk ~ 
 |k|^{2m\left( \frac{r}{2-r}\right) } \left|  h(|k|^2) \right|^{\left( \frac{r}{2-r}\right) }
 \right)^{\frac{2}{r}-1}
 \\
 \le
 C(h) t^{-\left(\frac{n}{2}+ m\left( \frac{r}{2-r}\right) \right)\left( \frac{2}{r}-1 \right)}
 =
 C(h) t^{-2\sigma_{r,m}}.
\end{multline*}
The convergence of the above integral with a time dependent domain is guaranteed by \eqref{hASSUME}.  The case $r=2$ then follows directly from 
$$
\sup_{|k|\leq 1} 
 |k|^{2m}
 h(|k|^2t)
 \le 
 C(h) t^{-m}
=
 C(h) t^{-2\sigma_{2,m}}.
$$
This completes the proof of \eqref{ineqW}.

\subsection*{Acknowledgments}  
 The author gratefulfly thanks Keya Zhu for her careful and thorough reading of this paper.

\end{document}